\documentclass[11pt,a4paper]{amsart}
\usepackage{amsmath,amssymb,amsthm}
\usepackage{ascmac}
\usepackage{color}

\usepackage[utf8]{inputenc}
\theoremstyle{definition}
 \newtheorem{dfn}{Definition}
 \newtheorem{remark}[dfn]{Remark}

\theoremstyle{plain}
\newtheorem{thm}[dfn]{Theorem}
 
 \newtheorem{lem}[dfn]{Lemma}
 \newtheorem{cor}[dfn]{Corollary}

\numberwithin{equation}{section}

\newcommand{\bu}{{\mathbf u}}

\newcommand{\bv}{{\mathbf v}}
\newcommand{\bw}{{\mathbf w}}
\newcommand{\ba}{{\mathbf a}}

\newcommand{\bff}{{\mathbf f}}

\newcommand{\bG}{{\mathbf G}}

\newcommand{\bI}{{\mathbf I}}

\newcommand{\bS}{{\mathbf S}}
\newcommand{\bT}{{\mathbf T}}
\newcommand{\bU}{{\mathbf U}}

\newcommand{\dv}{{\rm div}\,}

\newcommand{\BA}{{\mathbb A}}

\newcommand{\BR}{{\mathbb R}}
\newcommand{\BC}{{\mathbb C}}

\newcommand{\BN}{{\mathbb N}}
\newcommand{\BG}{{\mathbb G}}

\newcommand{\BI}{{\mathbb I}}

\newcommand{\BF}{{\mathbb F}}

\newcommand{\BZ}{{\mathbb Z}}

\newcommand{\CA}{{\mathcal A}}

\newcommand{\CD}{{\mathcal D}}

\newcommand{\CL}{{\mathcal L}}

\newcommand{\CR}{{\mathcal R}}
\newcommand{\CS}{{\mathcal S}}

\newcommand{\CH}{{\mathcal H}}

\newcommand{\bg}{{\bold g}}
\newcommand{\bh}{{\bold h}}

\newcommand{\pd}{\partial}

\newcommand{\HS}{\BR^N_+}
\newcommand{\WS}{\BR^N}

\renewcommand{\d}{\,\mathrm{d}}

\begin{document}
\title[]{$L_1$ approach to the compressible viscous fluid flows 
in the half-space}

\author[]{Jou chun Kuo}
\address{(J. Kuo) {School of Science and Engineering},
Waseda University \\
3-4-5 Ohkubo Shinjuku-ku, Tokyo, 169-8555, Japan }		
\email{kuojouchun@asagi.waseda.jp} 
\author[]{Yoshihiro Shibata}
\address{(Y. Shibata) 
{ Professor Emeritus of Waseda University \\
3-4-5 Ohkubo Shinjuku-ku, Tokyo, 169-8555, Japan. \\
Adjunct faculty member in the Department of Mechanical Engineering 
and	Materials Science, University of Pittsburgh, USA
}}
\email{yshibata325@gmail.com}


\thanks{The second author is  partially supported by JSPS KAKENHI 
Grant Number 22H01134}


\date{\today}   

\maketitle


\begin{abstract}
In this paper, we prove the local well-posedness for the Navier-Stokes equations describing the motion of
isotropic barotoropic compressible viscous fluid flow with 
non-slip boundary conditions, where the half-space $\HS = \{x=(x_1, \ldots, x_N) \in \BR^N \mid 
x_N>0\}$ ($N \geq 2$) is  the fluid domain.   
The density part of our solutions belongs to  $W^1_1((0, T), B^s_{q,1}(\HS))
\cap L_1((0, T), B^{s+1}_{q,1}(\HS))$ and the velocity part  of our solutions  
$W^1_1((0, T), B^{s}_{q,1}(\HS)^N) \cap L_1((0, T), B^{s+2}_{q,1}(\HS))$, 
where $B^\mu_{q,1}(\HS)$ denotes the inhomogeneous Besov
space on $\HS$. Namely, we solve the equations in  
the $L_1$ in time and $B^{s+1}_{q,1}(\HS) \times B^s_{q,1}(\HS)^N$ in space maximal regularity
framework. 
We use Lagrange transformation to eliminate the convection term $\bv\cdot\nabla\rho$ and 
we use an analytic semigroup approach.  
We only assume  the strictly positiveness of initial mass density. 
An essential assumption is that  $-1+N/q \leq  s < 1/q $ and $N-1 < q < \infty$.
Here,  $N/q$ is the crucial 
order to obtain  $\|\nabla \bu\|_{L_\infty} \leq C\|\nabla\bu\|_{B^{N/q}_{q,1}}$. 
\end{abstract}

\section{Introduction}

Let $1 < q < \infty$ and $-1+N/q \leq s < 1/q$, where $N\geq 2$ is the space dimension. 
In this paper, we use the  $L_1$--$B^{s+1}_{q,1}\times B^s_{q,1}$ maximal regularity framework to show
 the local well-posedness of the Navier-Stokes
equations describing the isotropic motion of the compressible viscous fluid flows
in the half-space. Let
$$\HS=\{x=(x_1, \ldots, x_N) \in \BR^N \mid x_N > 0\}, 
\quad
\pd\HS= \{x=(x_1, \ldots, x_N) \in \BR^N \mid x_N = 0\}.
$$
The equations considered in this paper read as
\begin{equation}\label{ns:1}\left\{\begin{aligned}
\rho_t + \dv(\rho\bv) = 0&&\quad&\text{in $\HS\times(0, T)$}, \\
\rho(\bv_t + \bv\cdot\nabla\bv) - \alpha\Delta \bv - \beta\nabla\dv\bv
+ \nabla P(\rho) = 0&&\quad&\text{in $\HS\times(0, T)$}, \\
\bv|_{\pd\HS} = 0, \quad(\rho, \bv) =(\rho_0, \bv_0)
&&\quad&\text{in $\HS$}.
\end{aligned}\right.\end{equation}
Here, 
$\alpha$ and $\beta$ denote respective
the viscosity coefficients and the second viscosity 
coefficients satisfying the conditions
\begin{equation}\label{assump:1.1}
\alpha > 0, \quad \alpha+\beta >0,
\end{equation}
and 
$P(\rho)$ is a smooth function defined on $(0, \infty)$ satisfying $P'(\rho) > 0$, 
that is,  the barotropic fluid is considered. 
\par
The main result of this paper is the following theorem. 
\begin{thm}\label{thm:1} Let $N-1 < q < \infty$ and $-1+N/q \leq  s < 1/q$. 
Let $\rho_*$ be a positive constant describing the mass density of the reference body,
and let $\tilde\eta_0 \in B^{s+1}_{q,1}(\HS)$.  Set $\eta(x) = \rho_* + \tilde\eta_0(x)$.
Assume that there exist two positive constants $\rho_1 < \rho_2$ such that 
 \begin{equation}\label{assump:0}
\rho_1 < \rho_* < \rho_2, \quad \rho_1 < P'(\rho_*) < \rho_2, \quad 
\rho_1 < \eta_0(x) < \rho_2, \quad  
\rho_1 < P'(\eta_0(x)) < \rho_2
\quad(x \in \overline{\HS}).
\end{equation}
 \par 
Then, 
there exist  small numbers $T>0$ and $\sigma>0$  such that 
for any initial data 
$\rho_0= \rho_*+\tilde\rho_0$ with $\tilde\rho_0
 \in B^{s+1}_{q,1}(\HS)$ and $\bv_0 \in B^s_{q,1}(\HS)$, 
problem \eqref{ns:1} admits unique solutions $\rho$ and $\bv$
satisfying the regularity conditions:
\begin{equation}\label{main.reg} \begin{aligned}
\rho-\rho_0 &\in L_1((0, T), B^{s+1}_{q,1}(\HS)) \cap W^1_1((0, T), B^s_{q,1}(\HS)),
\\
\bv &\in L_1((0, T), B^{s+2}_{q,1}(\HS)^N) \cap W^1_1((0, T), B^s_{q,1}(\HS)^N)
\end{aligned}\end{equation}
provided that $\|\tilde\rho_0-\tilde\eta_0\|_{B^{s+1}_{q,1}(\HS)}.
\leq \sigma$
\end{thm}
\begin{remark}\label{rem:1}
The condition  $-1 + N/q \leq s < 1/q$ requires that $-1+N/q < 1/q$.  Thus, the condition
$N-1 < q$ is necessary for our argument.  On the other hand, the requirement of $s < 1/q$
comes from our linear theory. To use the Abidi-Paicu-Haspot theory for the Besov
space estimate of the products of functions (cf. Lemma 8 in Sect. 2 below), we have to assume that
$-N/q' < s < N/q$.  Since we assume that $-1+N/q \leq s$, this condition is 
satisfied, because$ -N/q' < -1+N/q$, that is $N > 1$. 
\end{remark}
R. Danchin and R. Tolksdorf \cite{DT22} proved the local and global well-posedness 
of equations \eqref{ns:1} in the $L_1$ in time and $B^{N/q}_{q,1} \times B^{N/q -1}_{q,1}$ in space 
maximal regularity framework for some $q \in (2, min(4, 2N/(N-2))$, 
and the main assumption is that the fluid domain is bounded.  
In particular, they consider only the case where $s=N/q$ in our notation for thier local well-posedness
theory. 
To obtain  the $L_1$ in time  maximal regularity of solutions to the linearized equations,
so called Stokes equations in the compressible fluid flow case, in \cite{DT22} 
they used their extended version of 
Da Prato and Grisvard theory \cite{DG}, which was a first result concerning $L_1$ maximal regularity
for  continuous analytic semigroups.  In \cite{DT22} , they assumed that the fluid domain is bounded,
which  seems to be necessary to obtain the linear theory for Lam\'e equations cf. \cite[Sect. 3]{DT22}
in their argument  \par
The final goal of our study is 
to solve equations \eqref{ns:1}  if the fluid domain is a general $C^2$ class domain. 
If the fluid domain is the whole space, a number of results have been estabilished 
\cite{CD10, D08, AP07, H11} and references  given therein.   
Thus, our interest is in the initial boundary value problem case.
As a first step of our study, in this paper we consider equations in the half-space, 
namely the model problem for the initial boundary value problem.

\subsection{Problem Reformulation}\label{sec.1.1}
To prove Theorem \ref{thm:1}, it is advantageous to transfer 
equations \eqref{ns:1} to equations in  Lagrange coordinates.
In fact, 
the convection term $\bv\cdot\nabla\rho$ in the material derivative disappears
in the equations of Lagrange coordinates. \par
Let $\bu(x, t)$ be the velocity field in  Lagrange coordinates: 
$x=(x_1, \ldots, x_N)$ and we consider  Lagrange transformation:
$$y = X_\bu(x, t) := x + \int^t_0 \bu(x, \tau)\,\d\tau,
$$
where equations \eqref{ns:1} are written in Euler coordinates: $y=(y_1, \ldots, 
y_N)$. We assume that 
\begin{equation}\label{assump:2}
\Bigl\|\int^T_0 \nabla \bu(\cdot, \tau)\,\d\tau\Bigr\|_{L_\infty(\HS)}
\leq c_0
\end{equation}
with some small constant $c_0>0$, and then for each $t \in (0, T)$, 
the map: $X_\bu(x, t)  = y$ is a 
$C^1$ diffeomorphism from $\HS$ onto $\Phi(\HS)$ under the assumption that 
$\bu \in L_1((0, T), B^{s+2}_{q,1}(\HS)^N)$
with $-1 + N/q \leq s < 1/q$ (cf. Danchin et al \cite{DHMT}).  Moreover,
using an argument due to Str\"ohmer \cite{Str:89}, we have $\Phi(\HS) = \HS$, 
and so as a conclusion, $\Phi(\HS)$ is a $C^1$ diffeomorphism from 
$\HS$ onto $\HS$.\par 
We shall drive equations in Lagrange coordinates. Let 
$\BA_\bu$ is the Jacobi matrix of transformation: $y = X_\bu$, that is
$$\BA_\bu = \frac{\pd x}{\pd y} =(\frac{\pd y}{\pd x})^{-1} 
=\Bigl(\BI + \int^t_0\nabla \bu(x, \tau)\,\d\tau\Bigr)^{-1}
= \sum_{j=0}^\infty \Bigl(\int^t_0\nabla \bu(x, \tau)\,\d\tau\Bigr)^j, $$
which is well-defined under the smallness assumption \eqref{assump:2}, 
where $\BI$ denotes the $N\times N$ identity matrix. 
We have the following 
well-known formulas:
\begin{equation}\label{trans:1}\begin{aligned}
\nabla_y & = \BA_{\bu}^\top \nabla_x, \quad
\dv_y (\,\cdot\,)  = \BA_{\bu}^\top \colon \nabla_x (\,\cdot\,)
= \dv_x(\BA_{\bu} (\,\cdot \, )), \\
\nabla_y \dv_y (\, \cdot \,) & = 
\BA_{\bu}^\top \nabla_x ((\BA_{\bu}^\top - \BI) \colon \nabla_x (\, \cdot \, ))
+ \BA_{\bu}^\top \nabla_x \dv_x (\, \cdot \,), \\
\Delta_y (\,\cdot\,) &= \dv_y\nabla_y(\, \cdot \,) 
= \dv_x(\BA_{\bu} \BA_{\bu}^\top\nabla_x (\, \cdot \,))
= \dv_x((\BA_{\bu} \BA_{\bu}^\top-\BI)\nabla_x (\cdot)) + \Delta_x (\, \cdot \,).
\end{aligned}\end{equation}
 Transformation law \eqref{trans:1}
transforms the system of equations \eqref{ns:1} into the following system of equations: 
\begin{equation}\label{ns:2}\left\{\begin{aligned}
\pd_t\rho+\rho\dv\bu = F(\rho, \bu)& 
&\quad&\text{in $\HS\times(0, T)$}, \\
\rho\pd_t\bu - \alpha\Delta \bu  -\beta\nabla\dv\bu 
+  \nabla P(\rho)
 = 
 \bG(\rho, \bu)& &\quad&\text{in $\HS\times(0, T)$}, \\
\bu|_{\pd\HS} =0, \quad (\rho, \bu)|_{t=0} = (\rho_0, \bu_0)
& &\quad&\text{in $\HS$}.
\end{aligned}\right.\end{equation}
Here, we have set
\begin{equation}\label{term:1}\begin{aligned}
F(\rho, \bu) & =  \rho((\BI-\BA_\bu):\nabla\bu) 
\\
\bG(\rho, \bu)& = 
(\BI-(\BA_\bu^\top)^{-1})(\rho\pd_t\bu -\alpha\Delta\bu)  
 +\alpha(\BA_\bu^\top)^{-1}\dv((\BA_\bu\BA_\bu^\top-\BI):\nabla\bu)\\
&+\beta\nabla((\BA_\bu^\top-\BI):\nabla\bu).
\end{aligned}\end{equation}

For equations \eqref{ns:2}, we shall prove the following theorem.
\begin{thm}\label{thm:2} Let $N-1 < q < \infty$ and $-1+N/q \leq s < 1/q$. 
Let $\eta_0 = \rho_* + \tilde\eta_0$ be a given initial data such that 
$\tilde\eta_0 \in B^{s+1}_{q,1}(\HS)$ and for some
 positive constants $\rho_1$ and $\rho_2$, the assumption \eqref{assump:0}
holds. 
Then, there exist constants $\delta>0$ and $T>0$ such that 
for any initial data $\rho_0 \in B^{s+1}_{q,1}(\HS)$
and $\bu_0 \in B^s_{q,1}(\HS)^N$,   problem \eqref{ns:2} admits 
 unique solutions $\rho$ and $\bu$ satisfying
the regularity conditions:
$$\rho-\rho_0 \in W^1_1((0, T), B^{s+1}_{q,1}(\HS)), 
\quad 
\bu \in L_1((0, T), B^{s+2}_{q,1}(\HS)^N) \cap 
W^1_1((0, T), B^s_{q,1}(\HS)^N). 
$$
provided that  $\|\rho_0 - \eta_0\|_{B^{s+1}_{q,1}(\HS)}\leq \sigma$.
\end{thm}
\subsection{$L_1$ theory for the Stokes equations}~ To prove Theorem \ref{thm:2}, the key issue
is the $L_1$ maximal regularity theorem for the linearized equations of \eqref{ns:2} at
initial mass density $\eta_0(x) = \rho_* + \tilde\eta_0(x)$ with $\tilde\eta_0(x) 
\in B^{s;1}_{q,1}(\HS)$, which read as
\begin{equation}\label{s:1}\left\{\begin{aligned}
\pd_t\Pi + \eta_0(x)\dv\bU &=F&\quad&\text{in $\HS\times(0, \infty)$}, \\
\eta_0(x)\pd_t\bU - \alpha\Delta\bU - \beta\nabla\dv\bU + \nabla(P'(\eta_0(x))\Pi)
& = \bG &\quad&\text{in $\HS\times(0, \infty)$}, \\
\bU|_{\pd\HS} = 0, \quad (\Pi, \bU)|_{t=0} = (\rho_0, \bu_0)&&\quad
&\text{in $\HS$}.
\end{aligned}\right.\end{equation}
We shall prove the following theorem in Sect. 3 below, which will be used to prove Theorem \ref{thm:2}.
\begin{thm}\label{thm:evol} Let $N-1 < q < \infty$, $-1+N/q \leq s < 1/q$, and $T>0$.  Assume that 
$\tilde \eta_0(x) \in B^{s+1}_{q,1}(\HS)$ and that the assumption \eqref{assump:0} holds. 
Then, there exist positive  constants $\gamma>0$ and $C>0$ such that 
for any initial data $(\rho_0, \bu_0)$ and right membes $(F, \bG)$ such that $(\rho_0, \bu_0) \in 
B^{s+1}_{q,1}(\HS) \times B^s_{q,1}(\HS)^N$, 
$$e^{-\gamma t}F \in L_1(\BR_+, B^{s+1}_{q,1}(\HS)), \quad e^{-\gamma t}\bG 
\in L_1(\BR_+, B^s_{q,1}(\HS)^N),$$
then  the initial boundary problem \eqref{s:1} admits unique solutions $(\Pi, \bU)$ with 
$$e^{-\gamma t}\Pi \in W^1_1(\BR_+, B^{s+1}_{q,1}(\HS)), \quad
e^{-\gamma t}\bU \in L_1(\BR_+, B^{s+2}_{q,1}(\HS)^N) \cap W^1_1(\BR_+, B^s_{q,1}(\HS)^N)
$$
possessing the estimate:
\begin{align*}
&\|e^{-\gamma t}(\Pi, \pd_t\Pi)\|_{L_1(\BR_+, B^{s+1}_{q,1}(\HS))} 
+ \|e^{-\gamma t}\bU\|_{L_1(\BR_+, B^{s+2}_{q,1}(\HS))} + \|e^{-\gamma t}
\pd_t \bU\|_{L_1(\BR_+, B^s_{q,1}(\HS))} \\
&\quad \leq C(\|(\rho_0, \bu_0)\|_{B^{s+1}_{q,1}(\HS)\times B^s_{q,1}(\HS)}
+ \|e^{-\gamma t}(F, \bG)\|_{L_1(\BR_+, B^{s+1}_{q,1}(\HS)\times B^s_{q,1}(\HS))}).
\end{align*}
Here and in the sequel,  we set $\BR_+ = (0, \infty)$, and
$$\|e^{-\gamma t}f\|_{L_1(\BR_+, X)} = \int_0^\infty e^{-\gamma t}\|f(\cdot, t)\|_X\,
\d t.
$$
\end{thm}
In order to prove Theorem \ref{thm:3}, we use the properties of solutions to the corresponding 
generalized resolvent problem:
\begin{equation}\label{s:2}\left\{\begin{aligned}
\lambda \rho+\eta_0\dv\bv = f& 
&\quad&\text{in $\HS$}, \\
\eta_0(x)\lambda\bv - \alpha\Delta \bv  -\beta\nabla\dv\bv 
+  \nabla(P'(\eta_0) \rho) = \bg& &\quad&\text{in $\HS$}, \\
\bv|_{\pd\HS} =0.
\end{aligned}\right.\end{equation}
To state our main result for equations \eqref{s:2}, we introduce 
a parabolic sector  $\Sigma_\mu$ defined by setting 
\begin{equation}\label{resol:0}
\Sigma_\mu = \{\lambda \in \BC\setminus\{0\} \mid |\arg\lambda| \leq \pi-\mu\}
\quad \Sigma_{\mu, \gamma} = \{\lambda \in \Sigma_\mu \mid |\lambda| \geq \gamma\}.
\end{equation}
 and functional spaces $\CH^s_{q,1}(\HS)$ and $\CD^s_{q,1}(\HS)$ and their norms defined by setting 
\begin{equation}\label{space:1} \begin{aligned}
\CH^s_{q,1}(\HS) & = \{(f, \bg) \in B^{s+1}_{q,1}(\HS) \times B^s_{q,1}(\HS)^N\}, \\
\CD^s_{q,1}(\HS) & = \{(\rho, \bu) \in B^{s+1}_{q,1}(\HS)\times B^{s+2}_{q,1}(\HS)^N 
\mid  \bu|_{\pd\HS} = 0\}, \\
\|(f, &\bg)\|_{\CH^s_{q,1}(\HS)} = \|f\|_{B^{s+1}_{q,1}(\HS)} + \|\bg\|_{B^s_{q,1}(\HS)}, \\
\|(f, &\bg)\|_{\CD^s_{q,1}(\HS)} = \|f\|_{B^{s+1}_{q,1}(\HS)} + \|\bg\|_{B^{s+2}_{q,1}(\HS)}, 
\end{aligned}\end{equation}
Then, we shall show the following theorem.
\begin{thm}\label{thm:3}
Let $1 < q < \infty$ and $-1 + N/q \leq  s < 1/q$.  Assume that $s$ satisfies 
\eqref{assump:0}. 
Let $\eta_0(x)= \rho_* + \tilde\eta_0(x)$ with $\tilde\eta_0(x) \in B^{s+1}_{q,1}(\HS)$. 
Then, the following three assertions hold. \par
\thetag1 ~ There exist 
constants $\gamma>0$ and $C$ 
such that for any $\lambda \in \Sigma_\mu+ \gamma$ and 
$(f, \bg) \in \CH^s_{q,1}(\HS)$, 
problem \eqref{s:2} admits a unique solution
$(\rho, \bv) \in \CD^s_{q,1}(\HS)$
posssessing the estimate:
\begin{equation}\label{est:1.0}
\|\lambda(\rho, \bv)\|_{\CH^s_{q,1}} 
+ \|(\lambda^{1/2}\bar\nabla, \bar\nabla^2)\bv\|_{B^s_{q,1}} \leq C\|(f, \bg)\|_{\CH^s_{q,1}(\HS)}
\end{equation}
for every $\lambda \in \Sigma_\mu + \gamma$. 
\par
\thetag2~Let $\sigma>0$ be a small number such that 
$-1+1/q < s-\sigma < s+\sigma < 1/q$. 
Then, there exist $\bv_1$ and $\bv_2$ such that $\bv_i \in B^{s+2}_{q,1}(\HS)$ $(i=1,2)$,
$\bv = \bv_1+\bv_2$, and there hold 
\begin{equation}\label{fundest.2**}\begin{aligned}
\|(\lambda, \lambda^{1/2}\bar\nabla, \bar\nabla^2)\bv_1\|_{B^s_{q,1}(\HS)}
&\leq C|\lambda|^{-\frac{\sigma}{2}}\|\bg\|_{B^{s+\sigma}_{q,1}(\HS)}, \\
\|(\lambda, \lambda^{1/2}\bar\nabla, \bar\nabla^2)\pd_\lambda\bv_1\|_{B^s_{q,1}(\HS)}
&\leq C|\lambda|^{-(1-\frac{\sigma}{2})}\|\bg\|_{B^{s-\sigma}_{q,1}(\HS)}
\end{aligned}\end{equation} 
for every $\lambda \in \Sigma_\mu + \gamma$ and $\bg \in C^\infty_0(\HS)$, as well as 
\begin{equation}\label{fundest.3**}\begin{aligned}
\|(\lambda, \lambda^{1/2}\bar\nabla, \bar\nabla^2)\bv_2\|_{B^s_{q,1}(\HS)}
\leq C|\lambda|^{-1}\|(f, \bg)\|_{\CH^s_{q,1}(\HS)}, \\
\|(\lambda, \lambda^{1/2}\bar\nabla, \bar\nabla^2)\pd_\lambda\bv_2\|_{B^s_{q,1}(\HS)}
\leq C|\lambda|^{-2}\|(f, \bg)\|_{\CH^s_{q,1}(\HS)}
\end{aligned}\end{equation}
as for any $\lambda \in \Sigma_\mu + \gamma$
and $(f, \bg) \in \CH^s_{q,1}(\HS)$. 
\par
\thetag3~ 
There exist
constants $\gamma$ and $C$ such that for every $\lambda \in \Sigma_\mu + \gamma$
and $(f, \bg) \in \CH^s_{q,1}(\HS)$, there hold 
\begin{equation}\label{rho:1}\begin{aligned}
\|\rho\|_{B^{s+1}_{q,1}(\HS)} &\leq C|\lambda|^{-1}
\|(f, \bg)\|_{\CH^s_{q,1}(\HS)},  \\
\|\rho\|_{B^{s+1}_{q,1}(\HS)} &\leq C|\lambda|^{-2}
\|(f, \bg)\|_{\CH^s_{q,1}(\HS)}.
\end{aligned}\end{equation}
In the statement of \thetag1,  \thetag2 and \thetag3, 
the constants $\gamma$ and $C$ depend on $\rho_*$ and $\|\tilde\eta_0\|_{B^{s+1}_{q,1}}$.
\end{thm}
\par
To conclude this subsection, we explain the relationship between 
Theorem \ref{thm:evol} and Theorem \ref{thm:3} with a simple example.  
Let $X$ and $Y$ be two Banach spaces such that $X$ is continuously embedded into 
$Y$ and let $A$ be a linear closed operator from $X$ into $Y$. We assume that $A$
generates a $C_0$ analytic semigroup $\{T(t)\}_{t\geq0}$ on $Y$ associated with
an evolution equation
$$\pd_t u - Au = 0 \quad\text{for $t>0$}, \quad u|_{t=0} = f. $$
Here, $\pd_t$ denotes the derivative with respect to $t$.   According to 
semigroup theory cf. \cite{YK},  we have
\begin{alignat} 2\label{semigroup:1}
\|\pd_tT(t)f\|_Y &\leq Ce^{\gamma t}t^{-1}\|f\|_Y&\quad&\text{ for $f \in Y$}, \\
\label{semigroup:2}
\|\pd_tT(t)f \|_Y &\leq Ce^{\gamma t}\|f\|_X &\quad&\text{ for $f \in X$}, 
\end{alignat}
for some constants $C$ and $\gamma$. If we consider the case where $1 < p < \infty$,
then, we have
\begin{alignat*}2
\|\pd_tT(t)f\|_Y^p &\leq C^p e^{p\gamma t}t^{-p} \|f\|_Y^p &\quad&\text{for $f \in Y$}, \\
\|\pd_tT(t)f\|_Y^p &\leq C^p e^{p\gamma t} \|f\|_X^p &\quad&\text{for $f \in X$}, 
\end{alignat*}
Thus, choosing $\theta \in (0, 1)$ in such a way that $1 = (1-\theta) p $, that is 
$\theta = 1-1/p$, by real interpolation theory, we see that 
$$\Bigr\{\int^\infty_0 (e^{-\gamma t}\|\pd_tT(t)f\|_{Y})^p\,\d t\Bigr\}^{1/p}\leq C \|f\|_{(Y, X)_{1-1/p, p}}.
$$
(cf. \cite{SS08}). But, this idea does not apply to the $p=1$ case. Our idea is to find two spaces
$Y_\sigma$ and $Y_{-\sigma}$ corresponding to a small positive constant $\sigma$ such that 
$Y_\sigma \subset Y  \subset Y_{-\sigma}$ and there hold
\begin{alignat*}2
\|\pd_tT(t)f\|_Y &\leq C e^{\gamma t}t^{-(1-\sigma)} \|f\|_{Y_\sigma} &\quad&\text{for $f \in Y_\sigma$}, \\
\|\pd_tT(t)f\|_Y &\leq C e^{\gamma t}t^{-(1+\sigma)} \|f\|_{Y_{-\sigma}} &\quad&
\text{for $f \in Y_{-\sigma}$}. 
\end{alignat*}
Then, by real interpolation method, we have
$$\int^\infty_0 e^{-\gamma t}\|\pd_tT(t)\|_Y\,\d t \leq C\|f\|_{(Y_\sigma, Y_{-\sigma})_{1/2, 1}}
$$
If we may choose $Y$ in such a way that $Y = (Y_\sigma, Y_{-\sigma})_{1/2, 1}$, then we have
$$\int^\infty_0 e^{-\gamma t}\|\pd_tT(t)\|_Y\,\d t \leq C\|f\|_Y.
$$
This is $L_1$ in time maximal regularity for $\{T(t)\}_{t \geq 0}$. \par
In view of the spectral analysis point of view, we assume that there exists a $\gamma>0$ such that 
for any $\lambda \in \Sigma_\mu + \gamma$, there hold
\begin{alignat*}2
\|\lambda(\lambda\bI-A)f\|_Y &\leq C\|f\|_Y& \quad&\text{for any $f \in Y$}, \\
\|\lambda(\lambda\bI-A)f\|_Y &\leq C|\lambda|^{-\sigma}
\|f\|_{Y_\sigma}& \quad&\text{for any $f \in Y_\sigma$}, \\
\|\lambda \pd_\lambda(\lambda\bI-A)f\|_Y &\leq C|\lambda|^{-(1-\sigma)}
\|f\|_{Y_{-\sigma}}& \quad&\text{for any $f \in Y_{-\sigma}$}.
\end{alignat*}
Then, the first estimate implies the generation of $C_0$ analytic semigroup 
$\{T(t)\}_{t\geq 0}$, and the second and third estimates imply the corresponding 
estimates of $\{T(t)\}_{t\geq 0}$. Of couse, in the compressible Stokes equations case,
the situation is much comlicated, but the essential idea is the same as above.
 Detailed proof will be given in the proof of Theorem \ref{thm:t.2} in Sect. 3 below.

\subsection{Short History} The mathematical study of compressible viscous fluids has a long history since
1950's. In fact, the first result was a uniquness theorem prove by Graffi
\cite{Grafi} and Serrin \cite{Serrin}. A local in time existence theorem was proved by Nash \cite{Nash62},
Itaya \cite{Itaya71} and Vol'pert and Hudjaev \cite{VH} in $\BR^3$ in the H\"older continuous function space.
After these works by pioneers,  much study has been done with the development of modern 
mathematics. 
We do not aim to give an extensive list of references, but refer to the following references and references 
given therein only for  unique existence theorems of strong solutions. \par 
A local in time unique existence thoerem
was proved by Solonnikov \cite{Sol80} in $W^{2,1}_q$ with $N < q < \infty$, by Tani \cite{Tani77} in the 
H\"older spaces, by Str\"ohmer \cite{Str:89} with analytic semigroup approach
and by Enomoto and Shibata \cite{ES}   in the $L_p$-$L_q$ maximal regularity class, where $\CR$ boundedness
of solution operators have been used.  If the fluid domain is $\BR^N$, 
the local well-posedness was proved by Charve and Danchin \cite{CD10} 
in the $L_1$ in time framework. 
\par
A global well-posedness was proved by
Matsumura and Nishida \cite{MN80, MN83} by energy methods and refer
to the survey paper by Shibata and Enomoto \cite{SE18} for several extensions of Matsumura and 
Nishida's work and the optimal decay properties of solutions  in the whole space and exterior domains. 
The global well-posedness in the $L_1$ in time framework was proved by Danchin \cite{D08} and also 
see Charve and Danchin \cite{CD10}, Abidi and Paicu \cite{AP07} and Haspot \cite{H11}. 
The global well-posedness in the $L_q$ maximal regularity framework ($1 < q < \infty$) 
 was proved by
Mucha and Zajaczkowski \cite{MZ}  and 
 in the $L_p$ in time and $L_q$ in space 
maximal regularity framework ($1 < p, q < \infty$) by Shibata \cite{S22}.
 Kagei and Kobayashi \cite{KK02, KK05} proved the global well-posedness
with optimal decay rate in the half-space and by Kagei \cite{Kagei08} in the layer domain.
Periodic solutions were treated by Valli \cite{Vali83}, Tsuda \cite{Tsuda16} and 
references given  therein. 
\subsection{Notation} The symbols $\BN$, $\BR$ and $\BC$ denote the set of all natural numbers, real numbers and 
complex numbers. Set $\BN_0 = \BN \cup\{0\}$. Let 
$L_q(\Omega)$, $W^m_q(\Omega)$ and $B^s_{q,r}(\Omega)$ 
denote the standard 
Lebesgue space, Sobolev space, and Besov space definded on a domain $\Omega$ in $N$ dimensional Euclidean space
$\BR^N$, while $\|\cdot\|_{L_q(\Omega)}$, $\|\cdot\|_{W^m_q(\Omega)}$, and $\|\cdot\|_{B^s_{q,r}(\Omega)}$ denote
their norms.  For time interval $I$, $L_q(I, X)$ and $W^1_q(I, X)$ denote respective $X$-valued Lebsgue space and 
Sobolev space of order 1.  $W^\alpha_q(I, X) = (L_q(I, X), W^1_q(I, X))_{\alpha, q}$
for $\alpha \in (0, 1)$. Here, the real interpolation functiors are denoted by  $(\cdot, \cdot)_{\theta, r}$
 for $\theta \in (0, 1)$ and $1 \leq r \leq\infty$. 
For $1\leq q < \infty$, we write
$$
\|f\|_{L_q(I, X)} = \Bigl(\int_I\|f(t)\|_X^q\,\d t\Bigr)^{1/q}, \quad \|e^{-\gamma t}f\|_{L_q(I, X)}
= \Bigl(\int_I (e^{-\gamma t}\|f(t)\|_X)^q\,\d t\Bigr)^{1/q}.
$$
Let 
$BC^0(I, X)$ denote the set of all  $X$-valued bounded continuous 
functions defined on $I$.  For any integer $m \geq 1$, $BC^m(I, X)$ denotes the set of all $X$-valued 
bounded continuous functions whose derivatives exist and bounded in $I$  up to order $m$. Set 
$$\|f\|_{BC(I, X)} = \sup_{t \in I}\|f(t)\|_X, \quad \|f\|_{BC^m(I,X)} = \|f\|_{BC^0(I, X)}
+ \sum_{j=1}^m \sup_{t\in I}\|(D_t^jf)(t)\|_X.$$

For  differentiation with respect to space variables $x=(x_1, \ldots, x_N)$,
$D^\alpha f :=\pd_x^\alpha f = \pd^{|\alpha|}f/{\pd x_1^{\alpha_1} \cdots \pd x_N^{\alpha_N}}
$ for multi-index $\alpha=(\alpha_1, \ldots, \alpha_N)$ with  $|\alpha| = \alpha_1+\cdots + \alpha_N$. 
For the notational simplicity, we write $\nabla f = \{ \pd_x^\alpha f \mid|\alpha|=1\}$, $\nabla^2 f
= \{\pd_x^\alpha f\mid |\alpha|=2\}$, $\bar\nabla f=(f, \nabla f)$, $\bar\nabla^2 f = 
(f, \nabla f, \nabla^2 f)$. For a Banach space $X$,  $\CL(X)$ denotes
the set of all bounded linear operators from $X$ into itself and 
$\|\cdot\|_{\CL(X)}$ denotes its norm.    
Let $\bI$ denote the identity operator
and $\BI$ the $N\times N$ identity matrix.  
For any Banach space $X$ with norm $\|\cdot\|_X$, $X^N = \{\bff = (f_1, \ldots, f_N) \mid f_i \in X \enskip(i=1, \ldots, N)\}$
and $\|\bff\|_X = \sum_{i=1}^N \|f_i\|_X$. For a vector $\bv$ and a matrix $\BA$, $\bv^\top$ and $\BA^\top$ denote
respective the transpose of $\bv$ and the transpose of $\BA$.  \par
The letter $C$ denotes a generic constant and $C_{a,b,\cdots} = C(a, b, \cdots)$ denotes the constant depending on
quantities $a$, $b$, $\cdots$. $C$, $C_{a,b,\cdots}$, and $C(a, b, \cdots)$ may change from line to line. 
\section{Spectral Analysis}\label{sec.2}


In this section, we shall prove Theorem \ref{thm:3} as a perturbation from 
Lam\'e equations, which read as
\begin{equation}\label{lame:2}
\eta_0(x)\lambda \bv - \alpha\Delta\bv -\beta\nabla\dv\bv =
\bg \quad \text{in $\HS$},  \quad \bv|_{\pd\HS} =0
\end{equation}
for spectral parameter $\lambda \in \Sigma_\mu + \gamma$ with large enough $\gamma>0$.
Thus, we start with the existence theorem for equations \eqref{lame:2}.
\begin{thm}\label{thm:4.0}
Let $1 < q < \infty$ and  $-1 + 1/q < s < 1/q$.   Let $\sigma$ be a small positive number
such that $-1+1/q < s-\sigma < s < s+\sigma < 1/q$.  
 Let $\nu = s$ or $s\pm\sigma$.  
Assume that $\tilde\eta_0 \in B^{N/q}_{q,1}(\HS)$.    Moreover, $\eta_0(x)$ and $\rho_*$ satisfy
the assumptions \eqref{assump:0}.
Then, there exist constants 
$\gamma_1 > 0$ and $C>0$ depending on $s$, $\sigma$, and $\|\tilde\eta_0\|_{B^{N/q}_{q,1}(\HS)}$
such that for any $\lambda \in \Sigma_\mu+ \gamma_1$,
problem \eqref{lame:2} admits a unique solution
 $\bv \in B^\nu_{q,1}(\HS)^N$
satisfying  the estimate:
\begin{equation}\label{est:1.1}\begin{aligned}
\|(\lambda, \lambda^{1/2}\bar\nabla, \bar\nabla^2)
\bv\|_{B^{\nu}_{q,1}(\HS)} 
&\leq C\|\bg\|_{B^\nu_{q,1}(\HS)}, \\
\|(\lambda, \lambda^{1/2}\bar\nabla, \bar\nabla^2)
\pd_\lambda \bv\|_{B^{\nu}_{q,1}(\HS)} 
&\leq C|\lambda|^{-1}\|\bg\|_{B^\nu_{q,1}(\HS)}.
\end{aligned}\end{equation}
Moreover, for any $\lambda \in \Sigma_\mu + \gamma_1$ and $\bg \in C^\infty_0(\HS)^N$,
there holds
\begin{align}\label{est:1.1*}
\|(\lambda, \lambda^{1/2}\bar\nabla, \bar\nabla^2)\bv\|_{B^s_{q,1}(\HS)}
& \leq C(1+\|\tilde\eta_0\|_{B^s_{q,1}(\HS)})|\lambda|^{-\frac{\sigma}{2}}\|\bg\|_{B^{s+\sigma}_{q,1}(\HS)}
\end{align} as well as 
\begin{equation}\label{est:1.1**}
\|(\lambda, \lambda^{1/2}\bar\nabla, \bar\nabla^2)\pd_\lambda\bv\|_{B^s_{q,1}(\HS)}
+\|\bv\|_{B^s_{q,1}(\HS)}
 \leq C(1+\|\tilde\eta_0\|_{B^s_{q,1}(\HS)})|\lambda|^{-(1-\frac{\sigma}{2})}\|\bg\|_{B^{s-\sigma}_{q,1}(\HS)}.
\end{equation}
\end{thm}
\begin{remark} $C^\infty_0(\HS)$ is dense in $B^\nu_{q,1}(\HS)$ provided that $-1+1/q < \nu
< 1/q$ and $1 < q < \infty$.
\end{remark}
Before starting the proof of Theorem \ref{thm:4.0}, we show a lemma concerning the Besov norm estimates
of the product of functions. To this end, we start with following lemma.
\begin{lem}\label{lem:APH} Let $1 < q < \infty$ and $\nu \in \BR$.   If the condition $|\nu|
< N/q$ for $q \geq 2$ holds and otherwise the condition $-N/q' < \nu < N/q$ holds 
then for any $u \in B^\nu_{q,1}(\HS)$ and $v \in B^{N/q}_{q, \infty}(\HS) \cap L_\infty(\HS)$, 
there holds 
\begin{equation}\label{besovprod:1}
\|uv\|_{B^\nu_{q,1}(\HS)} \leq C_\nu\|u\|_{B^\nu_{q,1}(\HS)}\|v\|_{B^{N/q}_{q,1}(\HS)}.
\end{equation}
for some constant $C>0$ independent of $u$ and $v$. 
\end{lem}
\begin{proof}   For a proof, refer to \cite[Cor. 2.5]{AP07} and \cite[Cor. 1]{H11}.
\end{proof}
{\bf Proof of  Theorem \ref{thm:4.0}.} To prove Theorem \ref{thm:4.0}, 
we shall construct an approximate solution 
for each point $x_0 \in \overline{\HS}$. Let $\nu = s$ or $s\pm\sigma$.    Recall that $\eta_0(x) = \rho_*
+\tilde \eta_0(x)$ and $\tilde\eta_0 \in  B^{N/q}_{q,1}(\HS)$.
To construct an approximation solution, we use a theorem for unique existence of solutions
of  the constant coefficient Lam\'e equations which read as
\begin{equation}\label{fund:1}
\gamma_0\lambda\bv -\alpha\Delta \bv - \beta\nabla\dv\bv 
=\bg\quad\text{in $\HS$}, \quad 
\bv|_{\pd\HS}=0.
\end{equation}
From Kuo \cite{Kuo23} the following theorem follows.
\begin{thm}\label{thm:kuo} Let $1 < q < \infty$ and $-1+1/q < \nu < 1/q$. 
Assume that $\alpha$ and $\beta$ are 
constants satisfying the conditions:
\begin{equation}\label{assump:1}
\alpha >0, \quad \alpha + \beta>0.\end{equation}
Moreover, we assume that there exist positive constants $M_1$ and $M_2$ such that 
$$M_1 \leq \gamma_0  \leq M_2.$$ 
Then, there exists a $\gamma_K > 0$ independent of $\gamma_0$
 such that for any $\lambda \in 
\Sigma_\mu + \gamma_K$ and  $\bg \in B^\nu_{q,1}(\HS)^N$, 
problem \eqref{fund:1} admits a unique solution
$\bv \in B^{\nu+2}_{q,1}(\HS)$ satisfying
the estimate:
\begin{equation}\label{fundest.1} 
\|(\lambda, \lambda^{1/2}\bar\nabla,
\bar\nabla^2)\bv\|_{B^\nu_{q,1}(\HS)} \leq C\|\bg\|_{B^\nu_{q,1}(\HS)}
\end{equation}
for some constant $C$.  \par
Moreover, let $-1+1/q < s < 1/q$ and let $\sigma>0$ be a small positive constant such that 
$-1+1/q < s-\sigma < s < s+\sigma < 1/q$.  Then, for any $\lambda \in \Sigma_\mu + \gamma_K$ and 
$\bg \in C^\infty_0(\HS)^N$
a solution $\bv \in B^{s\pm\sigma+2}_{q,1}(\HS)^N \cap B^{s+2}_{q,1}(\HS)^N$ 
of equations \eqref{fund:1} satisfies 
the following estimates: 
\begin{align}
\|(\lambda, \lambda^{1/2}\bar\nabla, \bar\nabla^2)\bv\|_{B^s_{q,1}(\HS)}
&\leq C|\lambda|^{-\frac{\sigma}{2}} \|\bg\|_{B^{s+\sigma}_{q,1}(\HS)}, 
\label{fundest.2} \\
\|(\lambda, \lambda^{1/2}\bar\nabla, \bar\nabla^2)\pd_\lambda\bv\|_{B^s_{q,1}(\HS)}
&\leq C|\lambda|^{-(1-\frac{\sigma}{2})} \|\bg\|_{B^{s-\sigma}_{q,1}(\HS)}.
\label{fundest.3}
\end{align} 
\par
Here, the constants $\gamma_K$ and $C$ depend on $M_1$,  $M_2$, and $\nu$, 
but  independent
of $\gamma_0$  as far as the assumption \eqref{assump:1} holds. 
\end{thm}
\begin{remark} (1) The same assertions hold for the  whole space problem: 
\begin{equation}\label{fund:2}\begin{aligned}
\gamma_0\lambda\bv -\alpha\Delta \bv - \beta\nabla\dv\bv 
=\bg&&\quad&\text{in $\WS$}.
\end{aligned}\end{equation}
(2) The reason why we choose $-1+1/q < s < 1/q$ is that $C^\infty_0(\HS)$ is dense 
in $B^s_{q,1}(\HS)$, cf. \cite[pp.368--369]{M74}, \cite[p.132]{M76}, and \cite[Theorems 2.9.3, Theorem 2.10.3]{Tbook78}.
\\
(3) For any $\lambda \in \Sigma_\mu + \gamma$ and $\bg \in B^s_{q,1}(\HS)$, there holds
\begin{equation}
\label{fundest.3*}
\|\bv\|_{B^s_{q,1}(\HS)} \leq C|\lambda|^{-(1-\frac{\sigma}{2})}\|\bg\|_{B^{s-\sigma}_{q,1}(\HS)}.
\end{equation}
In fact, for any $\bg\in B^\nu_{q,1}(\HS)^N$, we define an operator $\CS_0(\lambda)\bg$
by $\bv = \CS_{\gamma_0}(\lambda)\bg$, where $\nu=s$ or $s\pm\sigma$, 
and $\bv$ is a unique solution of equations \eqref{fund:1}.
By Theorem \ref{thm:kuo}, $\CS_{\gamma_0}(\lambda)$ is well-defined and 
a $B^{\nu+2}_{q,1}(\HS)$ valued holomorphic function with respect
to $\lambda \in\Sigma_\mu + \gamma_K$.    Differentiating
equations \eqref{fund:1} with respect to $\lambda$, we have
\begin{equation}\label{fund:1*}
\gamma_0\lambda \pd_\lambda \bv- \alpha\Delta \pd_\lambda\bv - \beta\nabla\dv\pd_\lambda\bv
= -\gamma_0\bv \quad\text{in $\HS$}, \quad \pd_\lambda\bv|_{\pd\HS} = 0.
\end{equation}
Thus, we have $\pd_\lambda\bv = -\gamma_0\CS_{\gamma_0}(\lambda)\bv 
= -\gamma_0\CS_{\gamma_0}(\lambda)\CS_{\gamma_0}(\lambda)\bg$. 
Let $D^{\nu+2}_{q,1}(\HS) = \{\bu \in B^{\nu+2}_{q,1}(\HS) \mid
\bu|_{\pd\HS}=0\}$. Since $\CS_{\gamma_0}(\lambda)$ is a surjective map from $B^\nu_{q,1}(\HS)$ onto 
$D^{\nu+2}_{q,1}(\HS)$, and so the inverse map $\CS_{\gamma_0}(\lambda)^{-1}$
 exists and it is a surjective
map from $D^{\nu+2}_{q,1}(\HS)$ onto $B^\nu_{q,1}(\HS)$.  Thus, $\bv
= -\gamma_0^{-1}\CS_{\gamma_0}(\lambda)^{-1}\pd_\lambda\bv$. By \eqref{fundest.3}, we have
$$\|\bv\|_{B^s_{q,1}(\HS)} \leq C\|\bar\nabla^2\pd_\lambda\bv\|_{B^s_{q,1}} 
\leq C|\lambda|^{-(1-\frac{\sigma}{2})}\|\bg\|_{B^{s-\sigma}_{q,1}(\HS)},
$$
which shows \eqref{fundest.3*}.  From this consideration it follows that 
\eqref{fundest.3} and \eqref{fundest.3*} are equivalent.
\end{remark}
\begin{proof} When $\gamma_0=1$, by a result due to Kuo \cite{Kuo23} 
there exist positive constants $C$ and $\tilde\gamma$ such that 
the existense of solutions and \eqref{fundest.1}--\eqref{fundest.3} hold
for $\lambda \in \Sigma_{\mu, \tilde\gamma}$. 
Here, 
the constants $C$ and $\tilde\gamma>0$ depend only on $\alpha$ and $\beta$.   
 When $\gamma_0\not=1$,  the existence of solutions and estimates
\eqref{fundest.1}--\eqref{fundest.3} hold,  replacing $\lambda$ with $\gamma_0\lambda$,
provided that  $\gamma_0\lambda \in \Sigma_{\mu, \tilde\gamma}$. Since
$M_1 \leq \gamma_0 \leq M_2$, we see that $M_1|\lambda| \leq |\gamma_0\lambda| \leq M_2|\lambda|$. 
Thus, choosing $\tilde\gamma_K = \tilde\gamma M_1^{-1}$, we see that
$\gamma_0\lambda \in \Sigma_{\mu, \tilde\gamma}$ if
$\lambda \in \Sigma_{\mu, \tilde\gamma_K}$. From this consideration,
Theorem \ref{thm:kuo} follows from the $\gamma_0=1$ case by choosing 
$\gamma_K$ so large that $\Sigma_\mu + \gamma_K \subset \Sigma_{\mu, \tilde\gamma_K}$. 
Here, the constants $C$ and $\gamma_K$ depend on $\alpha$, $\beta$, $M_1$ and $M_2$.
\end{proof}

We continue the proof of Theorem \ref{thm:4.0}. First we consider the case where $x_0 \in \pd\HS$. 
We write 
$$B_d(x_0) = \{x \in \BR^N \mid |x-x_0| \leq d\}, 
\quad B_d = B_d(0).$$
 Let  
$\varphi \in C^\infty_0(B_2(0))$ and $\psi \in C^\infty_0(B_3(0))$
such that 
 $\varphi(x)=1$ for 
$x \in B_1(0)$ and $\psi(x) =1$ for $x \in B_{2}(0)$
 and set $\varphi_{x_0,d}(x) = \varphi((x-x_0)/d)$ and $\psi_{x_0,d}(x)=
\psi((x-x_0)/d)$.  Notice that $\varphi_{x_0,d}(x)=1$ for $x \in B_d(x_0)$ 
and $\varphi_{x_0,d}(x)=0$ for $x \not\in B_{2d}(x_0)$ and that $\psi_{x_0,d}(x) = 1$ on
${\rm supp}\,\varphi_{x_0,d}$ and $\psi_{x_0,d}(x)=0$ for $x \not\in B_{3d}(x_0)$.
In particular, $\varphi_{x_0,d}\psi_{x_0,d} = \varphi_{x_0,d}$.  \par
Let $\bv \in B^s_{q,1}(\HS)^N$ be a solution of equations:
\begin{equation}\label{st:1}
\eta_0(x_0)\lambda\bv - \alpha\Delta\bv - \beta\nabla\dv\bv
 = \bg\quad\text{in $\HS$}, \quad 
\bv|_{\pd\HS} =0.
\end{equation}
For simplicity, we omit $\HS$ and $N$ for the description of function spaces and their norms like
$B^\nu_{q,1}=B^{\nu}_{q,1}(\HS)^N$ and $\|\cdot\|_{B^\nu_{q,1}} = \|\cdot\|_{B^\nu_{q,1}(\HS)}$
in what follows unless confusion may occur. We define an operator 
$\bT_{x_0}(\lambda)$ acting on $\bg \in B^\nu_{q,1}$ by
$\bv = \bT_{x_0}(\lambda)\bg$.
By \eqref{assump:0}, $\rho_1 < \eta_0(x_0) < \rho_2$, and so 
 by Theorem \ref{thm:kuo} there exist constants
$C$ and $\gamma_{K,1}$ independent of $x_0$ such that 
\begin{equation}\label{est:f1}
\|(\lambda, \lambda^{1/2}\bar\nabla, \bar\nabla^2)\bT_{x_0}(\lambda)\bg\|_{B^\nu_{q,1}}
\leq C\|\bg\|_{B^\nu_{q,1}}
\end{equation}
for every $\lambda \in \Sigma_\mu + \gamma_{K,1}$. 
 Let $A_{x_0} = \eta_0(x_0)
+ \psi_{x_0}(x)(\eta_0(x)-\eta_0(x_0))$.
And then, $\bv$ satisfies the following
equations:
\begin{equation}\label{st:2}
A_{x_0}\lambda\bv - \alpha\Delta\bv - \beta\nabla\dv\bv
 = \bg
+ \bS_{x_0}(\lambda)\bg\quad\text{in $\HS$}, \quad 
\bv|_{\pd\HS} =0.
\end{equation}
Here, we have set 
\begin{align*}
\bS_{x_0}(\lambda)\bg & = 
\psi_{x_0, d}(x)(\eta_0(x)-\eta_0(x_0))\lambda\bv.
\end{align*}

We now estimate $\psi_{x_0,d}(\eta_0(x_0)-\eta_0(x))\lambda\bv$.  
Note that $\eta_0(x)-\eta_0(x_0) = \tilde\eta_0(x)-\tilde\eta_0(x_0)$. 
By Lemma \ref{lem:APH}, we have
\begin{equation}\label{fundest:1}
\|\psi_{x_0,d}(\cdot)(\eta_0(x_0)-\eta_0(\cdot))\lambda\bv(\cdot)\|_{B^{\nu}_{q,1}} \leq 
C\|\psi_{x_0,d}(\cdot)(\tilde\eta_0(x_0)-\tilde\eta_0(\cdot))
\|_{B^{N/q}_{q,1}}\|\lambda\bv\|_{B^\nu_{q,1}}. 
\end{equation}
To estimate $\|\psi_{x_0,d}(\cdot)(\eta_0(x_0)-\eta_0(\cdot))\|_{B^{N/q}_{q,1}}$, 
we use the following lemma  due to Danchin-Tolksdorf \cite[Proposition B.1]{DT22}.
\begin{lem}\label{Prop:B.1} Let $f \in B^{N/q}_{q,1}(\HS)$ for some $1 \leq q \leq \infty$.  
Then, 
$$
\lim_{d\to0}\|\psi_{x_0, d}(\cdot)(f(\cdot)-f(x_0))\|_{B^{N/q}_{q,1}(\HS)}=0\
\quad\text{uniformly with respect to $x_0$}.
$$
\end{lem}
By Lemma \ref{Prop:B.1}, for any $\delta > 0$ there exists a $d>0$ such that 
\begin{equation}\label{small:0.1}
\|\psi_{x_0,d}(\cdot)(\eta_0(x_0)-\eta_0(\cdot))\|_{B^{N/q}_{q,1}} \leq \delta
\end{equation}
Notice that the choice of distance $d$ is independent of $x_0$. 
From \eqref{fundest:1} and \eqref{small:0.1}, it follows that 
\begin{equation}\label{est:2}\begin{aligned}
\|\bS_{x_0}(\lambda)\bg\|_{B^\nu_{q,1}}
\leq C\delta\|\lambda\bv\|_{B^\nu_{q,1}}.
\end{aligned}\end{equation}
Choosing $d>0$ so small that $C\delta \leq 1/2$, 
we have $\|\bS_{x_0}\|_{\CL(B^\nu_{q,1})} \leq 1/2$.   
Thus, 
the inverse $(\bI + \bS_{x_0}(\lambda))^{-1}$ of the operator $\bI + \bS_{x_0}(\lambda)$
exists and $\|\bI + \bS_{x_0}(\lambda))^{-1}\|_{\CL(B^\nu_{q,1})} \leq 2$, 
where $\bI$ is the identity operator on $ B^\nu_{q,1}$.
Recalling the operator $\bT_{x_0}(\lambda)$ is defined 
by $\bv = \bT_{x_0}(\lambda)\bg$, 
and setting $\bw_{x_0} = \bT_{x_0}(\lambda)(\bI + \bS_{x_0}(\lambda))^{-1}\bg$, 
by \eqref{est:f1} we see that 
$\bw_{x_0}$ satisfies equations:
\begin{equation}\label{st:3*}
A_{x_0}\lambda\bw_{x_0} - \alpha\Delta\bw_{x_0} - \beta\nabla\dv\bw_{x_0}
 = \bg\quad\text{in $\HS$}, \quad 
\bw_{x_0}|_{\pd\HS} =0, \end{equation}
as well as the estimate
\begin{equation}\label{est:m1}
\|(\lambda, \lambda^{1/2}\bar\nabla, \bar\nabla^2)\bw_{x_0}\|_{B^\nu_{q,1}}
\leq C\|(\bI+\bS_{x_0}(\lambda))^{-1}\bg\|_{B^\nu_{q,1}}
\leq 2C\|\bg\|_{B^\nu_{q,1}}
\end{equation}
for every $\lambda \in \Sigma_\mu + \gamma_{K,1}$, where $C$ is independent of 
$d$.\par 
Finally,  we set $\bv_{x_0} = \varphi_{x_0,d}\bw_{x_0}$.  Since 
$\psi_{x_0,d}\varphi_{x_0,d} = \varphi_{x_0,d}$, we have 
$A_{x_0}\varphi_{x_0,d} = \eta_0(x)\varphi_{x_0,d}$.
From \eqref{st:3*} it follows that
\begin{equation}\label{st:4*}
\eta_0(x)\lambda\bv_{x_0} - \alpha\Delta\bv_{x_0} - \beta\nabla\dv\bv_{x_0}
 = \varphi_{x_0}\bg
+ \bU_{x_0}(\lambda)\bg\quad\text{in $\HS$}, \quad 
\bv_{x_0}|_{\pd\HS} =0, 
\end{equation}
where we have set
\begin{align*}
\bU_{x_0}(\lambda)\bg & = -\alpha((\Delta\varphi_{x_0,d})\bw_{x_0} + 2(\nabla\varphi_{x_0,d})\nabla\bw_{x_0})
-\beta(\nabla((\nabla\varphi_{x_0,d})\cdot\bw_{x_0}) + (\nabla\varphi_{x_0,d})\dv\bw_{x_0}).
\end{align*}
From \eqref{est:m1}, we see that 
\begin{equation}\label{pert:2}\begin{aligned}
\|(\lambda, \lambda^{1/2}\bar\nabla, \bar\nabla^2)\bv_{x_0}\|_{B^\nu_{q,1}}
\leq C_d\|\bg\|_{B^\nu_{q,1}},
\end{aligned}\end{equation}
as well as 
\begin{equation}\label{remainder:1}
\|\bU_{x_0}(\lambda)\bg\|_{B^\nu_{q,1}} \leq Cd^{-2}|\lambda|^{-1/2}\|\bg\|_{B^\nu_{q,1}}
\end{equation}
for every $\lambda \in \Sigma_\mu + \gamma_{K,1}$ and $0<d<1$. 
Here,  $C$ is a constant independent of $x_0$ and $d \in (0, 1)$.

\par
Next, we pick up $x_1 \in \HS$ and we choose 
$d_1>0$ such that $B_{3d_1}(x_1) \subset \HS$.
Let $\tilde \bg$ be a suitable extension of  $\bg$ to $\BR^N$ such that
$\tilde\bg|_{\HS} = \bg$ and $\|\tilde \bg\|_{B^\nu_{q,1}(\BR^N)} \leq C\|\bg\|_{B^\nu_{q,1}}$. 
Let $\varphi_{x_1, d_1}(x) = \varphi((x-x_1)/d_1)$ and $\psi_{x_1, d_1}(x) = \psi((x-x_1)/d)$.
 Analogously  to \eqref{pert:2} and 
\eqref{remainder:1}, if we choose $d_1>0$ small enough, there exist a  
$\bw_{x_1} \in B^{s+2}_{q,1}(\WS)^N$ satisfying equations
\begin{equation}\label{st:4*}
A_{x_1}\lambda\bw_{x_1} - \alpha\Delta\bw_{x_1} - \beta\nabla\dv\bw_{x_1}
= \tilde \bg\quad\text{in $\BR^N$},
\end{equation}
where $A_{x_1} = \eta_0(x_1) + \psi_{x_1, d_1}(\eta_0(x)
-\eta_0(x_1))$, and the estimate:
\begin{equation}\label{est:3*}
\|(\lambda, \lambda^{1/2}\bar\nabla, \bar\nabla^2)\bw_{x_1}\|_{B^\nu_{q,1}(\BR^N)}
\leq C\|\bg\|_{B^\nu_{q,1}}
\end{equation}
for any $\lambda \in \Sigma_\mu + \gamma_{K,1}$.  Let $\bv_{x_1}=
\varphi_{x_1}\bw_{x_1}$ and then $\bv_{x_1}$ satisfies equations:
\begin{equation}\label{st:4}
\eta_0(x)\lambda\bv_{x_1} - \alpha\Delta\bv_{x_1} - \beta\nabla\dv\bv_{x_1}
= \varphi_{x_1}\bg
+ \bU_{x_1}(\lambda)\bg\quad\text{in $\HS$}, \quad 
\bv_{x_1}|_{\pd\HS} =0, 
\end{equation}
where we have set
\begin{align*}
\bU_{x_1}(\lambda)\bg & = -\alpha((\Delta\varphi_{x_1, d_1})\bw_{x_1} + 2(\nabla\varphi_{x_1, d_1})\nabla\bw_{x_1})
-\beta(\nabla((\nabla\varphi_{x_1, d_1})\cdot\bw_{x_1}) + (\nabla\varphi_{x_1, d_1})\dv\bw_{x_1}).
\end{align*}
Moreover, by \eqref{est:3*}, we have 
\begin{gather}\label{pert:3}
\|(\lambda, \lambda^{1/2}\bar\nabla, \bar\nabla^2)\bv_{x_1}\|_{B^\nu_{q,1}}
\leq Cd_1^{-2}\|\bg\|_{B^\nu_{q,1}}, \\
\|\bU_{x_1}\bg\|_{B^\nu_{q,1}}  \leq Cd_1^{-2}|\lambda|^{-1/2}\|\bg\|_{B^\nu_{q,1}}
\label{remainder:2}
\end{gather}
for every $\lambda \in \Sigma_\mu + \gamma_{K,1}$ and $d_1 \in (0, 1)$, 
where $C$ is a constant independent of $x_1$ and $d_1 \in (0, 1)$.  


Finally, we consider  the far field case. 
Let $\tilde\psi \in C^\infty(\BR)$ which equals to $1$ for 
$|x| \geq2$ and  $0$ for $|x| \leq 1$, and set $\psi_R(x) = \tilde\psi(x/R)$. 
Let 
 $\bv$ be a  solution of equations 
\begin{equation}\label{s:3*}
\rho_*\lambda\bv - \alpha\Delta \bv  -\beta\nabla\dv\bv 
 = \bg \quad\text{in $\HS$}, \quad 
\bv|_{\pd\HS} =0
\end{equation}
for any $\lambda \in \Sigma_\mu + \gamma_{K,1}$.  Here, notice that $\rho_0 < \rho_* < \rho_2$. 
We define an operator $\bT_R(\lambda)$ by $\bv
= \bT_R(\lambda)\bg$. By Theorem \ref{thm:kuo}, we have
\begin{equation}\label{est:2.2}
\|(\lambda, \lambda^{1/2}\bar\nabla, \bar\nabla^2)\bT_R(\lambda)\bg\|_{B^\nu_{q,1}}
\leq C\|\bg\|_{ B^\nu_{q,1}}.
\end{equation}
 Set 
$A_R= \rho_* + \psi_R(x)(\eta_0(x)-\rho_*) =\rho_* + \psi_R(x)\tilde\eta_0(x)$.
By \eqref{s:3*}, we have
\begin{equation}\label{s:4*}
A_R\lambda\bv - \alpha\Delta \bv  -\beta\nabla\dv\bv 
 = \bg + \bS_{R}(\lambda)\bg\quad\text{in $\HS$}, \quad
\bv|_{\pd\HS} =0,
\end{equation}
where we have set
\begin{align*}
\bS_{R}(\lambda)\bg & = \psi_R(x)\tilde\eta_0(x)\lambda\bv.
\end{align*}
By Lemma \ref{lem:APH}, 
we have
\begin{equation}\label{fundest:2}
\|\bS_R(\lambda)\bg\|_{B^\nu_{q,1}}
\leq C\|\psi_R\tilde\eta_0\|_{B^{N/q}_{q,1}}
\|\lambda\bv\|_{B^\nu_{q,1}}.
\end{equation}
For any $\delta >0$ there exists an $R$ such that 
\begin{equation}\label{small:2}
\|\psi_R\tilde\eta_0\|_{B^{N/q}_{q,1}}  \leq \delta.
\end{equation}
This fact follows from the following lemma, the idea of whose proof is completely the same 
as in the proof of \cite[Proposition B.1]{DT22}.
\begin{lem}\label{Prop.B2} Let $f \in B^{N/q}_{q,1}$ for some $1 \leq q \leq \infty$.  Then,
for any $\delta > 0$, there exists an $R > 1$ such that 
$$\|\psi_R f\|_{B^{N/q}_{q,1}}  < \delta.$$
\end{lem}
\begin{proof} 
Let $m$ be an integer such that $N/q < m$.  Notice that $W^m_q(\HS)$ is dense in
$B^{N/q}_{q,1}(\HS)$.  Thus, first we assume that $f \in W^m_q$.  Then, 
$\|f\|_{W^m_q} < \infty$ and $\|f\|_{L_q} < \infty$, which implies that 
for any $\delta>0$, there exists an $R > 0$ such that $\|f\|_{W^m_q(B_R^c)} < \delta$ and 
$\|f\|_{L_q(B_R^c)}  < \delta$.  Here, $B_R^c = \{x \in \BR^N \mid |x| \geq R\}$.  Thus, 
$\|\psi_Rf\|_{W^m_q} < \delta$ and $\|\psi_Rf\|_{L_q} < \delta$.  In fact, 
$$\|\psi_Rf\|_{W^m_q} \leq C_m\sum_{|\beta|\leq m}R^{-(m-|\beta|)}\|D^\beta f\|_{L_q(B_R^c)}
\leq C_m\|f\|_{W^m_q(B_R^c)}
$$
for any $R \geq 1$ with some constant $C_m$ depending only on $m$ and 
$D^\alpha\tilde\psi$ ($|\alpha| \leq m$). Thus, 
choosing $R>0$ larger if necessary, we have $\|\psi_Rf\|_{W^m_q} < \delta.$ \par

Since $\|\psi_R f\|_{B^{N/q}_{q,1}(\HS)} \leq C\|\psi_R f\|_{L_q}^{1-\frac{N}{mq}}\|\psi_R f
\|_{W^m_q}^{\frac{N}{mq}}$ with some constant $C$ independent of $R$ and $f$, we have
$$\|\psi_R f\|_{B^{N/q}_{q,1}} \leq C\delta.$$
If we choose $R \geq 1$ larger, we have
$$\|\psi_R f\|_{B^{N/q}_{q,1}} \leq \delta/2.$$
Now, in the case where $f \in B^{N/q}_{q,1}$, we choose $g \in W^m_{q,1}$ such that 
$$\|\psi_R(g-f) \|_{B^{N/q}_{q,1}} < C\|g-f\|_{B^{N/q}_{q,1}} <\delta/2.
$$
Here, $C$ is a constant indepenent of $R$.  Thus, choosing $R>0$ in such a way that 
$\|\psi_Rg\|_{B^{N/q}_{q,1}} < \delta/2$,  we have
$$\|\psi_R f\|_{B^{N/q}_{q,1}} \leq \|\psi_R(f-g)\|_{B^{N/q}_{q,1}} 
+ \|\psi_Rg\|_{B^{N/q}_{q,1}}
< \delta.$$
This completes the proof of Lemma \ref{Prop.B2}.

\end{proof}

Combining \eqref{fundest:2} and \eqref{small:2} implies 
\begin{equation}\label{small:2.3}
\|\bS_R(\lambda)\bg\|_{B^\nu_{q,1}} \leq C\delta\|\bg\|_{B^\nu_{q,1}}.
\end{equation}
Choosing $\delta>0$ in such a way that  $C\delta \leq 1/2$, we have 
$\|\bS_R(\lambda)\|_{\CL(B^\nu_{q,1})} \leq 1/2$, and so 
the inverse operator $(\bI+\bS_R(\lambda))^{-1}$ exists
and $\|(\bI + \bS_R(\lambda))^{-1}\|_{\CL(B^\nu_{q,1})} \leq 2$
for every $\lambda \in \Sigma_\mu + \gamma_{K,1}$.  
Thus, by \eqref{est:2.2} and \eqref{s:4*},  
$\bw_R = \bT_R(\lambda)(\bI + \bS_R(\lambda))^{-1}\bg 
\in B^\nu_{q,1}$ satisfies equations
\begin{equation}\label{s:5*}
A_R\lambda\bw_R- \alpha\Delta \bw_R  -\beta\nabla\dv\bw_R 
 = \bg  \quad \text{in $\HS$}, \quad 
\bw_R|_{\pd\HS} =0, 
\end{equation}
as well as the estimate: 
\begin{equation}\label{est:2.3}
\|(\lambda, \lambda^{1/2}\bar\nabla, \bar\nabla^2)\bw_R\|_{B^\nu_{q,1}}
\leq C\|(\bI+\bS_R)^{-1}\bg\|_{B^\nu_{q,1}}
\leq 2C\|\bg\|_{B^\nu_{q,1}}.
\end{equation}

Let $\tilde\varphi \in C^\infty(\HS)$ such that $\tilde\varphi(x) =1$
for $|x| \geq 3$ and $0$ for $|x| \leq 2$ and set $\varphi_R = \tilde\varphi(x/R)$. 
We have $\psi_R \varphi_R = \varphi_R$, and so setting
$\bv_R = \varphi_R\bw_R \in B^\nu_{q,1}(\HS)$, we see that $A_R\varphi_R\lambda\bv_R
 = \eta_0(x)\lambda\bv_R$.  Thus, 
by \eqref{s:5*} and \eqref{est:2.3}, we see that 
$\bv_R$ satisfies  the equations: 
\begin{equation}\label{s:6*}
\eta_0(x)\lambda\bv_R- \alpha\Delta \bv_R  -\beta\nabla\dv\bv_R 
 = \varphi_R\bg + \bU_{R}(\lambda)\bg
\quad\text{in $\HS$}, \quad 
\bv_R|_{\pd\HS} =0,
\end{equation}
as well as the estimate: 
\begin{equation}\label{est:2.4}
\|(\lambda, \lambda^{1/2}\bar\nabla, \bar\nabla^2)\bv_R\|_{B^\nu_{q,1}}
\leq C\|\bg\|_{B^\nu_{q,1}}
\end{equation}
for any $\lambda \in \Sigma_\mu + \gamma_{K,1}$.  Here,  we have set
\begin{align*}
\bU_{R}(\lambda)\bg & = -\alpha((\Delta\varphi_R(x))\bw_R
+2 (\nabla\varphi_R(x))\nabla\bw_R) - \beta(\nabla((\nabla\varphi_R(x))\cdot\bw_R)
+ (\nabla\varphi_R(x))\dv\bw_R).
\end{align*}
By \eqref{est:2.3}, we have 
\begin{equation}\label{remainder:3}
\|\bU_{R}(\lambda)\bg
\|_{B^{\nu}_{q,1}} \leq C|\lambda|^{-1/2}\|\bg\|_{B^\nu_{q,1}}.
\end{equation}
\par
Choose points  $x^0_j \in \pd\HS$ ($j=1, \ldots, L_0$), and $x^1_j \in 
\HS$ ($j=1, \ldots, L_1$) and diameters $d > d_1$ suitably such that 
$$\overline{\HS} \subset B^c_R \cup \bigcup_{j=1}^{L_0}
B_d(x^0_j) \cup \bigcup_{j=1}^{L_1} B_{d_1}(x^1_j).
$$ 
Let $\psi^0_0(x) = \psi_R(x)$, $\psi^0_j(x) = \varphi((x-x^0_j)/d)$, and
 $\psi^1_j(x) = \varphi((x-x^1_j)/d_1)$, and set
$$\Psi(x) = \psi^0_0(x) + \sum_{i=0}^1\sum_{j=1}^{L_i} \psi^i_j(x).$$
We see that $\Psi(x) \geq 1$ for every $x \in \overline{\HS}$ and $\Psi \in C^\infty(\overline{\HS})$.
 Set 
$$\varphi^0_0(x) = \psi^0_0(x)/\Psi(x), \quad
\varphi^i_j(x) = \psi^i_j(x)/\Psi(x).$$
Obviously, $\varphi^0_j \in C^\infty_0(B_{2d}(x^0_j))$, 
$\varphi^1_j \in C^\infty_0(B_{2d_1}(x^1_j))$,  
$\varphi^0_0(x) = 0$ for 
$|x| \leq 2R$, and 
$$\varphi^0_0(x) + \sum_{i=0}^1\sum_{j=1}^{L_i}
\varphi^i_j(x) = 1\quad\text{for $x \in \overline{\HS}$}.
$$
Let $\bv^i_j = \bv_{x^i_j} = \varphi^i_j\bw_{x^i_j}$, 
and $\bv^0_0 = \bv_R = \varphi^0_0\bw_R$. 
Set $\bv = \bv^0_0 + \sum_{i=0}^1\sum_{j=1}^{L_i}\bv^i_j$, and then 
\begin{equation}\label{eq:t.1}
\eta_0(x)\lambda\bv -\alpha\Delta\bv-\beta\nabla\dv\bv
  = \bg
+ \bU(\lambda)\bg\quad\text{in $\HS$}, \quad 
\bv|_{\pd\HS}=0.
\end{equation}
Here, we have set 
\begin{align*}
\bU(\lambda)\bg &= -\alpha((\Delta\varphi^0_0)\bw_R +2 (\nabla\varphi^0_0)\nabla
\bw_R) -\beta(\nabla((\nabla\varphi^0_0)\cdot\bw_R) + (\nabla\varphi^0_0)\dv\bw_R )\\
& -\sum_{i=0}^1\sum_{j=1}^{L_i}\{\alpha((\Delta\nabla\varphi^i_j)\bw_{x^i_j} + 
2(\nabla\varphi^i_j)\nabla\bw_{x^i_j}) 
+\beta(\nabla((\nabla\varphi^i_j)\cdot\bw_{x^i_j})+(\nabla\varphi^i_j)\dv\bw_{x^i_j}  )\}.
\end{align*}
By \eqref{pert:2}, \eqref{pert:3}, and \eqref{est:2.4}, we have
\begin{equation}\label{main:est.1}
\|(\lambda, \lambda^{1/2}\bar\nabla, \bar\nabla^2)\bv\|_{B^\nu_{q,1}}
\leq C\|\bg\|_{B^{\nu}_{q,1}}.
\end{equation}
By \eqref{remainder:1}, \eqref{remainder:2}, and \eqref{remainder:3}, 
we have
\begin{equation}\label{remain:1}
\|\bU(\lambda)\bg\|_{B^\nu_{q,1}}
\leq C|\lambda|^{-1/2}
\|\bg\|_{B^\nu_{q,1}}
\end{equation}
for any $\lambda \in \Sigma_\mu + \gamma_{K,1}$. 
Choosing $\gamma_1\geq \gamma_{K,1}$ so large that $C((\sin\mu)\gamma_1)^{-1/2} \leq 1/2$, we see that 
for any $\lambda \in \Sigma_\mu + \gamma_1$ 
 $(\bI+\bU(\lambda))^{-1}$ exists and 
$\|(\bI+\bU(\lambda))^{-1}\|_{\CL(B^\nu_{q,1})} \leq 2$. 
If we define an operator $\bT(\lambda)$ by
$\bT(\lambda)\bg = \bv$,  by \eqref{eq:t.1}
 $\bv = \bT(\lambda)(\bI+\bU(\lambda))^{-1}\bg$ satisfies
equations:
\begin{equation}\label{eq:t.2}
\eta_0(x)\lambda\bv -\alpha\Delta\bv-\beta\nabla\dv\bv
  = \bg\quad\text{in $\HS$}, \quad 
\bv|_{\pd\HS}=0.
\end{equation}
Moreover, by \eqref{main:est.1}, we have 
$$\|(\lambda, \lambda^{1/2}\bar\nabla, \bar\nabla^2)\bT(\lambda)(\bI+\bU(\lambda))^{-1}\bg\|_{B^\nu_{q,1}}
\leq C\|(\bI+\bU(\lambda))^{-1}\bg\|_{B^\nu_{q,1}}
\leq 2C\|\bg\|_{B^\nu_{q,1}}
$$
for any $\lambda \in \Sigma_\mu + \gamma_1$. 
 This completes the proof of \eqref{est:1.1}. \par
 Since $C^\infty_0(\HS)$ is dense in $B^\nu_{q,1}(\HS)$ whenever $-1+1/q < \nu < 1/q$.
 Thus, we may assume that $\bg \in C^\infty_0(\HS)$ below. 
 Problem \eqref{lame:2}
 admits a unique solution $\bv \in B^{\nu}_{q,1}(\HS)$
 satisfying the estimates:
 $$\|(\lambda, \lambda^{1/2}\bar\nabla, \bar\nabla^2)\bv\|_{B^\nu_{q,1}(\HS)}
 \leq C\|\bg\|_{B^\nu_{q,1}(\HS)}$$
 for $\nu=s$ and $\nu=s+\sigma$. This shows the first inequality in
 \eqref{est:1.1}. \par
  We now prove \eqref{fundest.2} and \eqref{fundest.3}. 
Since $C^\infty_0(\HS)$ is dense in $B^\nu_{q,1}(\HS)$, 
we may assume that $\bg \in C^\infty_0(\HS)^N$.
Notice that $\eta_0=\rho_*+\tilde\eta_0$.
 Applying \eqref{fundest.2} to 
 $$\rho_*\lambda \bv - \alpha\Delta\bv -\beta\nabla\dv\bv = \bg -\tilde\eta_0 \lambda\bv$$
 gives that 
 \begin{align*}
 \|(\lambda, \lambda^{1/2}\bar\nabla, \bar\nabla^2)\bv\|_{B^s_{q,1}}
 \leq C|\lambda|^{-\frac{\sigma}{2}}(\|\bg\|_{B^{s+\sigma}_{q,1}} 
 + \|\tilde\eta_0\lambda\bv\|_{B^{s+\sigma}_{q,1}}).
\end{align*}
By Lemma \ref{lem:APH}, we have 
$$\|\tilde\eta_0\lambda\bv\|_{B^{s+\sigma}_{q,1}} \leq C\|\tilde\eta_0\|_{B^{N/q}_{q,1}}
\|\lambda\bv\|_{B^{s+\sigma}_{q,1}} 
\leq C\|\tilde\eta_0\|_{B^{N/q}_{q,1}}\|\bg\|_{B^{s+\sigma}_{q,1}}.$$
Thus, we have \eqref{est:1.1*}. \par
 Then problem \eqref{lame:2} admits a unique solution 
$\bv \in B^{s+2}_{q,1}\cap B^{s+2-\sigma}_{q,1}$ satisfying the estimates:
\begin{align*}
 \|(\lambda, \lambda^{1/2}\bar\nabla, \bar\nabla^2)\bv\|_{B^\nu_{q,1}}
 \leq C\|\bg\|_{B^\nu_{q,1}}
\end{align*}
for $\nu=s$ and $\nu=s-\sigma$.  Differentiating equations \eqref{lame:2} with respect to 
$\lambda$ yields 
\begin{equation}\label{diflame:1}
\eta_0(x)\lambda \pd_\lambda\bv-\alpha\Delta \pd_\lambda \bv - \beta\nabla\dv\pd_\lambda\bv
= -\eta_0(x)\bv.
\end{equation}
Thus, by the first inequality in \eqref{est:1.1} and Lemma \ref{lem:APH},  we have
\begin{align*}
\|(\lambda, \lambda^{1/2}\bar\nabla, \bar\nabla^2)\pd_\lambda\bv\|_{B^\nu_{q,1}}
&\leq C\|\eta_0\bv\|_{B^\nu_{q,1}} \leq C(\rho_*+\|\tilde\eta_0\|_{B^{N/q}_{q,1}})\|\bv\|_{B^\nu_{q,1}}\\
&\leq C|\lambda|^{-1}(\rho_*+\|\tilde\eta_0\|_{B^{N/q}_{q,1}})\|\bg\|_{B^\nu_{q,1}}.
\end{align*}
This shows the second inequality in \eqref{est:1.1}.
Applying  \eqref{fundest.3*} to
$$\rho_*\lambda \bv - \alpha\Delta\bv -\beta\nabla\dv\bv = \bg -\tilde\eta_0 \lambda\bv$$
 gives that 
 \begin{align*}
 \|\bv\|_{B^s_{q,1}(\HS))}
 &\leq C|\lambda|^{-(1-\frac{\sigma}{2})}(\|\bg\|_{B^{s-\sigma}_{q,1}(\HS)} 
 + \|\tilde\eta_0\lambda\bv\|_{B^{s-\sigma}_{q,1}(\HS)}).
\end{align*}
By Lemma \ref{lem:APH}, we have
$$\|\tilde\eta_0\lambda\bv\|_{B^{s-\sigma}_{q,1}(\HS)} 
\leq C\|\tilde\eta_0\|_{B^{N/q}_{q,1}(\HS)}\|\lambda \bv\|_{B^{s-\sigma}_{q,1}(\HS)}
\leq C\|\tilde\eta_0\|_{B^{N/q}_{q,1}(\HS)}\|\bg\|_{B^{s-\sigma}_{q,1}(\HS)}.$$
Thus, we have 
\begin{equation}\label{diffest.1}
\|\bv\|_{B^s_{q,1}(\HS)} \leq C(1+\|\tilde\eta_0\|_{B^{N/q}_{q,1}(\HS)})
|\lambda|^{-(1-\frac{\sigma}{2})}\|\bg\|_{B^{s-\sigma}_{q,1}(\HS)}.
\end{equation}
To estimate $\pd_\lambda \bv$, we differentiate equations \eqref{lame:2} with respect to $\lambda$,
and then we have
$$\eta_0(x)\lambda\pd_\lambda\bv-\alpha\Delta\pd_\lambda\bv- \beta\nabla\dv\pd_\lambda\bv
= - \eta_0(x)\bv\quad\text{in $\HS$}, \quad \pd_\lambda\bv|_{\pd\HS}=0.$$
Thus, applying the estimate \eqref{est:1.1} gives that 
$$\|(\lambda, \lambda^{1/2}\bar\nabla, \bar\nabla^2)\pd_\lambda\bv\|_{B^s_{q,1}(\HS)}
\leq C\|\eta_0\bv\|_{B^s_{q,1}(\HS)}.$$
By  \eqref{diffest.1},   we have
$$\|\eta_0\bv\|_{B^s_{q,1}(\HS)} \leq (\rho_*+C\|\tilde\eta_0\|_{B^{N/q}_{q,1}(\HS)})
\|\bv\|_{B^s_{q,1}(\HS)} 
\leq C(\rho_* +\|\tilde\eta_0\|_{B^{N/q}_{q,1}(\HS)})
|\lambda|^{-(1-\frac{\sigma}{2})}\|\bg\|_{B^{s-\sigma}_{q,1}(\HS)}, $$
which implies that 
$$\|(\lambda, \lambda^{1/2}\bar\nabla, \bar\nabla^2)\pd_\lambda\bv\|_{B^s_{q,1}(\HS)}
\leq C|\lambda|^{-(1-\frac{\sigma}{2})}\|\bg\|_{B^{s-\sigma}_{q,1}(\HS)}.
$$
This completes the proof of Theorem \ref{thm:4.0}
\qed

Now, we consider  problem \eqref{s:2} of the Stokes system
and prove Theorem \ref{thm:3}.  
 We insert the relation: 
 $\rho = \lambda^{-1}(f-\eta_0\dv\bv)$ obtained 
from the first equation in \eqref{s:2} into the second equations.  Then, we have
\begin{equation}\label{SL:1}
\eta_0(x)\lambda\bv - \alpha\Delta \bv - \beta\nabla\dv\bv -
\lambda^{-1}\nabla(P'(\eta_0)\eta_0\dv\bv) = \bh
\quad\text{in $\HS$}, \quad \bu|_{\pd\HS} =0,
\end{equation}
where we have set $\bh = \bg - \lambda^{-1}\nabla(P'(\eta_0)f)$.  
In what follows, restore the notation of $\HS$ like $B^s_{q,1}(\HS)$, 
$\|\cdot\|_{B^s_{q,1}(\HS)}$ etc. 

As a first step to analyze equations \eqref{SL:1}, we shall prove the following
theorem.
\begin{thm}\label{thm:4.1}
Let $1 < q < \infty$ and $-1 + N/q \leq s < 1/q$.  Let $\sigma > 0$ be a small number such that 
$-1+1/q < s-\sigma < s < s+\sigma < 1/q$, and let $\nu=s$ or $s\pm\sigma$. 
 Let $\eta_0(x) = \rho_* + \tilde\eta_0(x)$ with 
$\tilde\eta_0\in B^{s+1}_{q,1}(\HS)$.
 Let $\gamma_1>0$  be the constant  given in
Theorem \ref{thm:4.0}. 
Then, there exist   $\gamma_2 \geq \gamma_1$ and an operator family 
$\CS(\lambda)$
such that $\CS(\lambda) \in {\rm Hol}\,(\Sigma_\mu+\gamma_2, \CL(B^s_{q,1}(\HS),
B^{s+2}_{q,1}(\HS)))$, for any $\lambda \in \Sigma_\mu+\gamma_2$ and $\bh \in B^\nu_{q,1}(\HS)$
$\bv=\CS(\lambda)\bh$ is a unique solution of equations \eqref{SL:1}, and there hold
\begin{align*}\|(\lambda, \lambda^{1/2}\bar\nabla, \bar\nabla^2)\CS(\lambda)\bh
\|_{B^s_{q,1}(\HS)} \leq C\|\bh\|_{B^s_{q,1}(\HS)}, \\
\|(\lambda, \lambda^{1/2}\bar\nabla, \bar\nabla^2)\pd_\lambda\CS(\lambda)\bh
\|_{B^s_{q,1}(\HS)} \leq C|\lambda|^{-1}\|\bh\|_{B^s_{q,1}(\HS)}.
\end{align*}
\par
Moreover, 
there are two operator families $\CS^i(\lambda) 
\in {\rm Hol}\,(\Sigma_\mu+\gamma_2, \CL(B^\nu_{q,1}(\HS), B^{\nu+2}_{q,1}(\HS)))$ 
$(i=1,2)$ such that
$\CS(\lambda) = \CS^1(\lambda)+\CS^2(\lambda)$, 
 \begin{align*}
\|(\lambda, \lambda^{1/2}\bar\nabla, \bar\nabla^2)\CS^1(\lambda)\bh
\|_{B^s_{q,1}(\HS)} &\leq C|\lambda|^{-\frac{\sigma}{2}}\|\bh\|_{B^{s+\sigma}_{q,1}(\HS)}, \\
\|(\lambda, \lambda^{1/2}\bar\nabla, \bar\nabla^2)\pd_\lambda\CS^1(\lambda)\bh
\|_{B^s_{q,1}(\HS)} &\leq C|\lambda|^{-(1-\frac{\sigma}{2})}\|\bh\|_{B^{s-\sigma}_{q,1}(\HS)}
\end{align*}
for any $\lambda \in \Sigma_\mu+ \gamma_2$ 
 and $\bh \in C^\infty_0(\HS)$, and
\begin{align*}
\|(\lambda, \lambda^{1/2}\bar\nabla, \bar\nabla^2)\CS^2(\lambda)\bh
\|_{B^s_{q,1}(\HS)} \leq C|\lambda|^{-1}\|\bh\|_{B^s_{q,1}(\HS)}, \\
\|(\lambda, \lambda^{1/2}\bar\nabla, \bar\nabla^2)\pd_\lambda\CS^2(\lambda)\bh
\|_{B^s_{q,1}(\HS)} \leq C|\lambda|^{-2}\|\bh\|_{B^s_{q,1}(\HS)}
\end{align*}
for any $\lambda \in \Sigma_\mu+ \gamma_2$
 and $\bh \in B^s_{q,1}(\HS)$. \par
 Here, the constants $\gamma_2$ and $C$ depend on $\rho_*$ and $\|\tilde\eta_0\|_{B^{s+1}_{q,1}}$.
\end{thm}
\begin{proof}
Since $C^\infty_0(\HS)$ is dense in $B^\nu_{q,1}(\HS)$ with $1 < q < \infty$ and
$-1+1/q < \nu < 1/q$.  Thus, we assume that $\bh \in C^\infty_0(\HS)$ in the sequel
if we do not mention the functional spaces which $\bh$ belongs to. 
First of all, we shall solve equations \eqref{SL:1} for $\lambda \in \Sigma_\mu+ \gamma_2$
with large $\gamma_2\geq \gamma_1$. 
By Lemma \ref{lem:APH}, 
$$
\|\nabla(P'(\eta_0)\eta_0\dv\bv)\|_{B^\nu_{q,1}}
\leq C(\|(P''(\eta_0)\eta_0 + P'(\eta_0))
(\nabla\tilde\eta_0)\dv\bv\|_{B^\nu_{q,1}}
+ \|P'(\eta_0)\eta_0 \nabla\dv\bv\|_{B^\nu_{q,1}}).
$$
We now use the following lemma for the Besov norm estimate of composite functions cf. 
\cite[Proposition 2.4]{H11} and \cite[Theorem 2.87]{BCD}.
\begin{lem}\label{lem:Hasp} Let $1 < q < \infty$. 
Let $I$ be an open interval of $\BR$.  Let $\omega>0$ and let $\tilde\omega$ be the smallest
integer such that $\tilde\omega \geq \omega$. Let $F:I \to \BR$ satisfy $F(0) = 0$ and 
$F' \in BC^{\tilde\omega}(I, \BR)$.  Assume that $v\in B^\omega_{q,r}$
has valued in $J \subset\subset I$.  Then, 
$F(v) \in B^\omega_{q,1}$ and there exists a constant $C$ depending only on 
$\nu$, $I$, $J$, and $N$, such that 
$$\|F(v)\|_{B^\omega_{q,1}} \leq C(1 + \|v\|_{L_\infty})^{\tilde\omega}
\|F'\|_{BC^{\tilde\omega}(I,\BR)}
\|v\|_{B^\omega_{q,1}}.$$
\end{lem}
Recalling that $\eta_0 = \rho_* + \tilde\eta_0$, we write
\begin{align*}
&(P''(\eta_0)\eta_0 + P'(\eta_0)) \\
&= (P''(\rho_*)+ \int^1_0P'''(\rho_* + \ell\tilde\eta_0)\,\d\ell \eta_0)
 (\rho_*+\tilde\eta_0) + P'(\rho_*) + \int^1_0P''(\rho_* + 
\ell\tilde\eta_0)\,\d\ell \tilde\eta_0\\
& = P''(\rho_*)\rho_*+ P'(\rho_*) + Q_1(\tilde\eta_0)
\end{align*}
where we have set
\begin{align*}
Q_1(u) & = \rho_*\int^1_0P'''(\rho_*+\ell u)\,\d\ell\, u
+ (P''(\rho_*)+ \int^1_0P'''(\rho_* + \ell u)\,\d\ell\,u)u
+ \int^1_0P''(\rho_* + \ell u)\,\d\ell\, u.
\end{align*}
In view of \eqref{assump:0}, $\rho_1-\rho_* < \tilde\eta_0(x) < \rho_2 - \rho_*$
and $\rho_1-\rho_* < 0 < \rho_2-\rho_*$. 
Thus, we may assume that 
there exists an $\kappa_0 > 0$ such that 
$\rho_1 < \rho_* - \kappa_0 < \rho_* + \kappa_0 < \rho_2$ for any $x \in \overline{\HS}$ and 
$|\eta_0(x)| <\kappa_0$ for any $x \in \overline{\HS}$.   
In particular, 
we may assume that
\begin{equation}\label{assump:2.1}
| \ell\tilde\eta_0(x)| < \kappa_0
\end{equation}
for any $\ell \in [0, 1]$ and $x \in \overline{\HS}$.
From this observation, we may assume that $Q_1(u)$ is defined for 
$u \in [-\kappa_0, \kappa_0]$ and $Q_1(0) = 0$.

By Lemmas \ref{lem:APH} and \ref{lem:Hasp}  we have 
\begin{align*}
&\|(P''(\eta_0)\eta_0+ P'(\eta_0))\nabla\tilde\eta_0
\dv\bv\|_{B^{s}_{q,1}(\HS)}\\
&\quad  \leq  C(|P''(\rho_*)\rho_*+P'(\rho_*)|\|\nabla\tilde \eta_0\|_{B^{s}_{q,1}(\HS)}
\|\dv\bv\|_{B^{N/q}_{q,1}(\HS)} \\
&\qquad 
+ \|Q_1(\tilde\eta_0)\|_{B^{N/q}_{q,1}(\HS)}\|\nabla\tilde\eta_0\|_{B^{s}_{q,1}(\HS)}
\|\dv\bv\|_{B^{N/q}_{q,1}(\HS)} \\
&\quad \leq C(\rho_*, \|\tilde\eta_0\|_{B^{s+1}_{q,1}(\HS)})\|\bv\|_{B^{s+2}_{q,1}(\HS)} .
\end{align*}
Here, we have use the assumption that $N/q \leq s+1$, and 
$C(\rho_*, \|\tilde\eta_0\|_{B^{s+1}_{q,1}(\HS)})$ denotes a constant depending on $\rho_*$,  
$\|\tilde\eta_0\|_{B^{s+1}_{q,1}(\HS)}$. 

Likewise, we write
$$P'(\eta_0)\eta_0 = (P'(\rho_*) + \int^1_0 P''(\rho_*+\ell\tilde\eta_0)\,
\d\ell \tilde\eta_0)(\rho_* + \tilde\eta_0)
= P'(\rho_*)\rho_* + Q_2(\tilde\eta_0),
$$
where we have set
$$Q_2(u) = \int^1_0P''(\rho_* + \ell u)\,\d\ell \, u \rho_* + 
(P'(\rho_*) + \int^1_0 P''(\rho_* + \ell u)\,\d\ell \, u)u.$$
$Q_2(u)$ is defined for  $u \in [-\kappa_0, \kappa_0]$ and  $Q_2(0) = 0$.
By Lemmas \ref{lem:APH} and  \ref{lem:Hasp}, we have
$$\|P'(\eta_0)\eta_0\nabla\dv\bv\|_{B^s_{q,1}(\HS)}
\leq C(|P'(\rho_*)\rho_* | + (1+\|\tilde\eta_0\|_{L_\infty(\HS)})^m\|\tilde\eta_0
\|_{B^{N/q}_{q,1}(\HS)})\|\bv\|_{B^{s+2}_{q,1}(\HS)}.
$$
for some integer $m \geq 1$. Therefore, we have
\begin{equation}\label{23.jn.20.3}
\|\nabla(P'(\eta_0)\eta_0\dv\bv)\|_{B^s_{q,1}(\HS)} \leq 
C(\rho_*, \|\tilde\eta_0
\|_{B^{s+1}_{q,1}(\HS)})\|\bv\|_{B^{s+2}_{q,1}(\HS)}.
\end{equation}
Choosing $\gamma_2\geq \gamma_1$ so large that 
$\gamma_2^{-1}C(\rho_*, \|\tilde\eta_0
\|_{B^{s+1}_{q,1}(\HS)}) \leq 1/2$, we have
$$
|\lambda|^{-1}\|\nabla(P'(\eta_0)\eta_0\dv\bv)\|_{B^s_{q,1}(\HS)}
\leq (1/2)\|\bv\|_{B^{s+2}_{q,1}(\HS)}
$$
for any $\lambda \in \Sigma_\mu + \gamma_2$. 
\par
Let $\bv$ be a unique solution of equations \eqref{lame:2} for $\bg$ and let 
$T(\lambda)$ be an operator defined by $\bv = T(\lambda)\bg$. For any $\bh \in B^\nu_{q,1}(\HS)$,
we insert $\bu = T(\lambda)\bh$ into equations \eqref{SL:1}, and then we have
\begin{align*}
&\eta_0(x)\lambda T(\lambda)\bh 
- \alpha \Delta T(\lambda)\bh - \beta\nabla\dv T(\lambda)\bh
-\lambda^{-1}\nabla(P'(\eta_0)\eta_0\dv T(\lambda)\bh) \\
&\quad = \bh -\lambda^{-1}\nabla(P'(\eta_0)\eta_0\dv T(\lambda)\bh)
\quad\text{in $\HS$}, \quad T(\lambda)\bh|_{\pd\HS}=0.
\end{align*}
If we set $R(\lambda) \bh = \nabla(P'(\eta_0)\eta_0\dv T(\lambda)\bh)$,
by \eqref{23.jn.20.3} and \eqref{est:1.1}, we have
\begin{equation}\label{est:r.1}
\|R(\lambda)\bh\|_{B^s_{q,1}(\HS)} \leq C\|\bh\|_{B^s_{q,1}(\HS)}, \quad
\|\pd_\lambda R(\lambda)\bh\|_{B^s_{q,1}(\HS)} \leq C|\lambda|^{-1}\|\bh\|_{B^s_{q,1}(\HS)},
\end{equation}
Choosing $\gamma_2 \geq \gamma_1$ so large that $\gamma_2^{-1}C \leq 1/2$, we see that 
the inverse operator $(\bI - \lambda^{-1}R(\lambda))^{-1} 
= \bI + \lambda^{-1}R(\lambda)\sum_{j=0}^\infty (\lambda^{-1}R(\lambda))^j$
exists as a bounded linear operator on $B^s_{q,1}(\HS)$. Thus, if we define $\CS(\lambda)$
by $\CS(\lambda)
= T(\lambda)(\bI - \lambda^{-1}R(\lambda))^{-1}$, then $\CS(\lambda)$
is a solution operator of equations \eqref{SL:1}, that is $\bv = \CS(\lambda)\bh$ is
a unique solution of equations \eqref{SL:1}.  Moreover, by \eqref{est:1.1} we have
\begin{equation}\label{est:2.1}
\|(\lambda, \lambda^{1/2}\bar\nabla, \bar\nabla^2)\CS(\lambda)\bh\|_{B^s_{q,1}(\HS)}
\leq C\|\bh\|_{B^s_{q,1}(\HS)}
\end{equation}
for any $\lambda \in \Sigma_\mu + \gamma_2$ and $\bh \in B^s_{q,1}(\HS)$. \par
Writing $(\bI - \lambda^{-1}R(\lambda))^{-1} 
= \bI + \lambda^{-1}R(\lambda)\sum_{j=0}^\infty (\lambda^{-1}R(\lambda))^j
=\bI + \lambda^{-1}R(\lambda)(\bI-\lambda^{-1}R(\lambda))^{-1}$
we set 
\begin{equation}\label{operator:2.1}\begin{aligned}
\CS^1(\lambda) = T(\lambda), \quad 
\CS^2(\lambda) =T(\lambda)\lambda^{-1}R(\lambda)
(\bI-\lambda^{-1}R(\lambda))^{-1}
\end{aligned}\end{equation}
Obviously, $\CS(\lambda) = \CS^1(\lambda) + \CS^2(\lambda)$.
By \eqref{est:1.1}, \eqref{est:1.1*} and \eqref{est:1.1**}, we have
\begin{align}
\|(\lambda, \lambda^{1/2}\bar\nabla, \bar\nabla^2)\CS^1(\lambda)\bh\|_{B^s_{q,1}(\HS)}
&\leq C\|\bh\|_{B^s_{q,1}(\HS)}, \label{est:2.2}\\
\|(\lambda, \lambda^{1/2}\bar\nabla, \bar\nabla^2)\pd_\lambda\CS^1(\lambda)\bh\|_{B^s_{q,1}(\HS)}
& \leq C|\lambda|^{-1}\|\bh\|_{B^s_{q,1}(\HS)} \label{est:2.3}
\end{align}
for any $\lambda \in \Sigma_\mu + \gamma_2$ and $\bh \in B^\nu_{q,1}(\HS)$ as well as 
\begin{align}
\|(\lambda, \lambda^{1/2}\bar\nabla, \bar\nabla^2)\CS^1(\lambda)\bh\|_{B^s_{q,1}(\HS)}
&\leq C|\lambda|^{-\frac{\sigma}{2}}\|\bh\|_{B^{s+\sigma}_{q,1}(\HS)},\label{est:2.4} \\
\|(\lambda, \lambda^{1/2}\bar\nabla, \bar\nabla^2)\pd_\lambda\CS^1(\lambda)\bh\|_{B^s_{q,1}(\HS)}
&\leq C|\lambda|^{-(1-\frac{\sigma}{2})}\|\bh\|_{B^{s-\sigma}_{q,1}(\HS)} \label{est:2.5}
\end{align}
for any $\lambda \in \Sigma_\mu + \gamma_2$ and $\bh \in C^\infty_0(\HS)$. \par
Moreover, by \eqref{est:1.1} and the first inequality of \eqref{est:r.1} we have
\begin{equation}\label{est:2.6}
\|(\lambda, \lambda^{1/2}\bar\nabla, \bar\nabla^2)\CS^2(\lambda)\bh\|_{B^\nu_{q,1}(\HS)}
\leq C|\lambda|^{-1}\|\bh\|_{B^{\nu}_{q,1}(\HS)}
\end{equation}
for any $\lambda \in \Sigma_\mu+\gamma_2$ and $\bh \in B^\nu_{q,1}(\HS)$. 
We write
\begin{align*}
\pd_\lambda \CS^2(\lambda) &=(\pd_\lambda T(\lambda))\lambda^{-1}R(\lambda)
(\bI-\lambda^{-1}R(\lambda)) \\
&+ T(\lambda)\pd_\lambda(\lambda^{-1}R(\lambda))
(\bI - \lambda^{-1}R(\lambda))^{-1}\\
&+ T(\lambda)(\lambda^{-1}R(\lambda)(\bI-\lambda^{-1}R(\lambda))^{-1}
(\pd_\lambda (\lambda^{-1}R(\lambda)))(\bI-\lambda^{-1}R(\lambda))^{-1}.
\end{align*}
Using the estimate
$\|\pd_\lambda T(\lambda)\bh\|_{B^s_{q,1}(\HS)} \leq C|\lambda|^{-1}\|\bh\|_{B^s_{q,1}(\HS)}$
as follows from \eqref{est:1.1} and \eqref{est:r.1}, we have
\begin{equation}\label{est:2.7}
\|(\lambda, \lambda^{1/2}\bar\nabla, \bar\nabla^2)
\pd_\lambda \CS^2(\lambda)\bh\|_{B^s_{q,1}(\HS)} \leq C|\lambda|^{-2}\|\bh\|_{B^s_{q,1}(\HS)}
\end{equation}
for any $\lambda \in \Sigma_\mu + \gamma_2$ and $\bh \in B^s_{q,1}(\HS)$.  Here, we have
used $|\lambda|^{-3} \leq (\gamma_2\sin \epsilon)^{-1}|\lambda|^{-2}$ in the estimate of the last term.
Combining \eqref{est:2.3} and \eqref{est:2.7} yields
\begin{equation}\label{est:2.8}
\|(\lambda, \lambda^{1/2}\bar\nabla, \bar\nabla^2)
\pd_\lambda \CS(\lambda)\bh\|_{B^s_{q,1}(\HS)}\leq C|\lambda|^{-1}\|\bh\|_{B^s_{q,1}(\HS)}
\end{equation}
for any $\lambda \in \Sigma_\mu + \gamma_2$ and $\bh \in B^\nu_{q,1}(\HS)$.
This completes the proof of Theorem \ref{thm:4.1}.

\end{proof}
{\bf Proof of Theorem \ref{thm:3}.}~ Recall the symbols defined in \eqref{space:1}, which will
be used below. 
Let $\bv = \CS(\lambda)(\bg-\lambda^{-1}\nabla(P'(\eta_0)f)$, 
and then $\bv$ is a unique solution of equations \eqref{SL:1}
with $\bh= \bg-\lambda^{-1}\nabla(P'(\eta_0)f)$. Using the formula $\CS(\lambda)
= \CS^1(\lambda) + \CS^2(\lambda)$, we divide $\bv$ as $\bv=\bv_1 + \bv_2$, where
\begin{align*}
\bv_1 &= \CS^1(\lambda)\bg, \quad
\bv_2 = \CS^2(\lambda)\bg - \lambda^{-1}\CS(\lambda)\nabla(P'(\eta_0)f).
\end{align*}
By Lemmas \ref{lem:APH} and \ref{lem:Hasp}, and the assumption: $N/q \leq s+1$,  we have
$$\|\nabla(P'(\eta_0)f)\|_{B^s_{q,1}(\HS)} \leq C(\rho_*, \|\tilde\eta_0\|_{B^{s+1}_{q,1}(\HS)})
\|f\|_{B^{s+1}_{q,1}(\HS)}.$$
Combining  \eqref{est:2.1}, \eqref{est:2.2}, \eqref{est:2.4}, \eqref{est:2.5}, \eqref{est:2.6}, 
\eqref{est:2.7}, \eqref{est:2.8}, we have 
\begin{equation}\label{est:2.10}\begin{aligned}
\|(\lambda, \lambda^{1/2}\bar\nabla, \bar\nabla^2)\bv\|_{B^s_{q,1}(\HS)} & \leq C
\|(f, \bg)\|_{\CH^s_{q,1}(\HS)}, \\
\|(\lambda, \lambda^{1/2}\bar\nabla, \bar\nabla^2)\pd_\lambda \bv\|_{B^s_{q,1}(\HS)} 
& \leq C|\lambda|^{-1}
\|(f, \bg)\|_{\CH^s_{q,1}(\HS)}, \\
\|(\lambda, \lambda^{1/2}\bar\nabla, \bar\nabla^2)\bv_2\|_{B^s_{q,1}(\HS)}
& \leq C|\lambda|^{-1}\|(f, \bg)\|_{\CH^s_{q,1}(\HS)}, \\
\|(\lambda, \lambda^{1/2}\bar\nabla, \bar\nabla^2)\pd_\lambda\bv_2\|_{B^s_{q,1}(\HS)}
& \leq C|\lambda|^{-2}\|(f, \bg)\|_{\CH^s_{q,1}(\HS)}
\end{aligned}\end{equation}
for any $\lambda \in \Sigma_\mu + \gamma_2$ and $(f, \bg) \in \CH^s_{q,1}(\HS)$.
Moreover,
\begin{equation}\label{est:2.9}\begin{aligned}
\|(\lambda, \lambda^{1/2}\bar\nabla, \bar\nabla^2)\bv_1\|_{B^s_{q,1}(\HS)} & \leq C
\|\bg\|_{B^{s+\sigma}_{q,1}(\HS)}, \\ 
\|(\lambda, \lambda^{1/2}\bar\nabla, \bar\nabla^2)\pd_\lambda\bv_1\|_{B^s_{q,1}(\HS)} & \leq C
\|\bg\|_{B^{s-\sigma}_{q,1}(\HS)},
\end{aligned}\end{equation}
for any $\lambda \in \Sigma_\mu + \gamma_2$ and $\bg \in C^\infty_0(\HS)^N$. 
\par 
Finally, define $\rho$ by  
$\rho=\lambda^{-1}(f-\eta_0\dv\bv)$.  By Lemmas \ref{lem:APH}
and  \ref{lem:Hasp} and $N/q \leq s+1$, we have  
\begin{align*}
\|\lambda\rho\|_{B^{s+1}_{q,1}(\HS)} &\leq C(\|f\|_{B^{s+1}_{q,1}(\HS)} 
+ C(\rho_*,  \|\tilde\eta_0\|_{B^{s+1}_{q,1}(\HS)})
\|\bv\|_{B^{s+2}_{q,1}(\HS)})
\\
& \quad \leq  C(\rho_*, \|\nabla\tilde\eta_0\|_{B^{s+1}_{q,1}})(\|f\|_{B^{s+1}_{q,1}(\HS)}
+ \|\bg\|_{B^s_{q,1}(\HS)})
\end{align*}
for every $\lambda \in \Sigma_\mu + \gamma_2$, which, combined with the first inequality in
\eqref{est:2.9}, implies the required resolvent estimate: 
$$\|\lambda(\rho, \bv)\|_{\CH^s_{q,1}(\HS)} 
+ \|(\lambda^{1/2}\bar\nabla, \bar\nabla^2)\bv\|_{B^s_{q,1}(\HS)} 
\leq C\|(f, \bg)\|_{\CH^s_{q,1}(\HS)}$$
for any $\lambda \in \Sigma_\mu+\gamma_2$ and $(f, \bg) \in \CH^s_{q,1}(\HS)$.\par
We now prove \eqref{fundest.2**} and \eqref{fundest.3**}. By Lemmas \ref{lem:APH} and \ref{lem:Hasp}
and the assumption: $N/q \leq s+1$, 
\begin{align*}
\|\rho\|_{B^{s+1}_{q,1}(\HS)}& \leq C|\lambda|^{-1}(\|f\|_{B^{s+1}_{q,1}(\HS)} + C(\rho_*,
\|\tilde\eta_0\|_{B^{s+1}_{q,1}(\HS)}\|\bv\|_{B^{s+2}_{q,1}(\HS)}) 
\leq C|\lambda|^{-1}\|(f, \bg)\|_{\CH^s_{q,1}(\HS)}
\end{align*}
for any $\lambda \in \Sigma_\mu + \gamma_2$ and $(f, \bg) \in \CH^s_{q,1}(\HS)$.
Differentiating the definition of $\rho$ with respect to $\lambda$ implies 
$$\pd_\lambda \rho= -\lambda^{-2}(f -\eta_0\dv\bv) + \lambda^{-1} \eta_0\dv\pd_\lambda\bv.$$
Since 
$$\|\eta_0\dv\pd_\lambda^\ell\bv\|_{B^{s+1}_{q,1}(\HS)} 
\leq C(\rho_*, \|\tilde\eta_0\|_{B^{s+1}_{q,1}(\HS)})\|\dv\pd_\lambda\bv\|_{B^{s+1}_{q,1}(\HS)}
\quad(\ell=0,1)
$$
as follows from Lemmas \ref{lem:APH} and \ref{lem:Hasp} and the assumption:
$N/q \leq s+1$, by the first two inequality in \eqref{est:2.10} we have
$$\|\pd_\lambda\rho\|_{B^{s+1}_{q,1}(\HS)} \leq C|\lambda|^{-2}\|(f, \bg)\|_{\CH^s_{q,1}(\HS)}$$
for any $\lambda \in \Sigma_\mu + \gamma_2$ and $(f, \bg) \in \CH^s_{q,1}(\HS)$.
This completes the proof of Theorem \ref{thm:3}. \qed

\section{$L_1$ semigroup}

In this section, we assume that $1 < q < \infty$, $-1+N/q \leq s < 1/q$, $\sigma>0$, and 
$-1/q < s-\sigma < s < s+\sigma < 1/q$. 
Let $\eta_0(x) = \rho_* + \tilde\eta_0(x)$ with $\tilde\eta_0(x) \in B^{s+1}_{q,1}(\HS)$
and $\eta_0(x)$ satisfies  the conditions \eqref{assump:0}. 
In the sequel, let $\mu \in (0, \pi/2) $ be  fixed and let $\gamma>0$ be a constant given in 
Theoerem \ref{thm:3}.  \par 
In this section, we consider evolution equations:
\begin{equation}\label{semi:1}\left\{\begin{aligned}
\pd_t\rho + \eta_0(x)\dv\bu &=F&\quad&\text{in $\HS\times(0, T)$}, \\
\eta_0(x)\pd_t\bu - \alpha\Delta\bu - \beta\nabla\dv\bu + \nabla(P'(\eta_0(x))\rho)
& = \bG &\quad&\text{in $\HS\times(0, T)$}, \\
\bu|_{\pd\HS} = 0, \quad (\rho, \bu) = (f, \bg)&&\quad
&\text{in $\HS$}.
\end{aligned}\right.\end{equation}
The corresponding generalized resolvent problem to  \eqref{semi:1} is 
equations \eqref{s:2}.   Let  $\CH^s_{q,1}(\HS)$ and $\CD^s_{q,1}(\HS)$ be the spaces defined 
in \eqref{space:1} while $\|\cdot\|_{\CH^s_{q,1}(\HS)}$ and $\|\cdot\|_{\CD^s_{q,1}(\HS)}$
are their norms defined also in \eqref{space:1}. 
Let $\CA$ be an operator  defined by  
\begin{equation}\label{op:1}
\CA(\rho, \bu) = (\eta_0\dv\bu, \,\,
\eta_0(x)^{-1}(-\alpha\Delta\bu -
\beta\nabla\dv\bu+ \nabla(P'(\eta_0(x)\rho)))
\end{equation}
for $(\rho, \bu) \in \CD^s_{q,1}(\HS)$. Then, problem \eqref{s:2}  reads as 
\begin{equation}\label{resol:0}
(\lambda\bI + \CA)(\rho, \bu)= (f, \eta_0(x)^{-1}\bg).
\end{equation}
Noticing that $\eta_0(x)^{-1} = \rho_*^{-1} -\tilde\eta_0(x)(\rho_*(\rho_*+\tilde\eta_0(x))^{-1}$,
we see that there exists a constant $c_0>0$ depending on $\rho_*$ and 
$\|\tilde\eta_0\|_{B^{s+1}_{q,1}(\HS)}$ such that 
$$c_0^{-1}\|\bg\|_{B^\nu_{q,1}(\HS)} \leq \|\eta_0^{-1}\bg\|_{B^\nu_{q,1}(\HS)} 
\leq c_0\|\bg\|_{B^\nu_{q,1}(\HS)}.
$$
for $\nu=s$ or $\nu=s\pm \sigma$. 
Thus, Theorem \ref{thm:3} holds for the equations \eqref{resol:0}, which replaces equations \eqref{s:2}.
and so  $\CA$ generates a continuous analytic semigroup 
$\{T(t)\}_{t\geq 0}$ and  solutions $\Pi$ and $\bU$ of equations \eqref{s:1} are given by
\begin{equation}\label{sol:1}
(\Pi, \bU) = T(t)(\rho_0, \bu_0) + \int^t_0T(t-s)(F(\cdot, s), \rho_0(\cdot)^{-1}\bG(\cdot, s))\,\d s. 
\end{equation}
\par
We now prove the $L_1$ in time maximal regularity of $ \{T(t)\}_{t\geq 0}$. 
The idea of our proof here is due to  Shibata \cite{S23}, cf also 
Kuo \cite{Kuo23} and  Shibata and Watanabe \cite{ SW1}.
 Let $T_1(t)$ and $T_2(t)$ denote the mass density part of $T(t)$ and the velocity part of 
$T(t)$.  Namely, $T(t)(f, \bg) = (T_1(t)(f, \bg), T_2(t)(f, \bg))$ and $\rho= T_1(t)(f, \bg)$ and 
$\bu = T_2(t)(f, \bg)$ 
satisfy equations \eqref{semi:1} with $F=\bG=0$. 
\begin{thm}\label{thm:t.2}
Let $1 < q < \infty$ and $-1+N/q \leq  s < 1/q$.  
Let $\tilde \eta_0(x) \in B^{s+1}_{q,1}(\HS)$ and $\eta_0(x) = \rho_* + \tilde\eta_0(x)$ satisfies the 
assumption \eqref{assump:0}. Let $\gamma>0$ be a constant given in Theorem \ref{thm:3}, 
which depends on $\rho_*$ and $\|\tilde\eta_0\|_{B^{s+1}_{q,1}(\HS)}$. 
Then, there exists a constant $C>0$ depending on $\rho_*$ and $\|\eta_0\|_{B^{s+1}_{q,1}(\HS)}$
such that for any $(f, \bg) \in \CH^s_{q,1}(\HS)$, there holds 
\begin{equation}\label{L1.1}
\int^\infty_0e^{-\gamma t}(\|(\pd_t, \bar\nabla^2)T_2(t)(f, \bg)\|_{B^s_{q,1}(\HS)}
+ \|(1, \pd_t)T_1(t)(f, \bg)\|_{B^{s+1}_{q,1}(\HS)})\,\d t
\leq C\|(f, \bg)\|_{\CH}.
\end{equation}
\end{thm}
\begin{proof}
Let $(\theta, \bv) = (\lambda+\CA)^{-1}(f, \bg)$, then $\theta \in 
B^{s+1}_{q,1}(\HS)$ and $\bv \in B^{s+2}_{q,1}(\HS)^N$ satisfy 
equations \eqref{resol:0}.  Since $C^\infty_0(\HS)$ is dense in $B^\nu_{q,1}(\HS)$ for 
$-1+1/q < \nu < 1/q$ with $1 < q < \infty$,  we may assume that 
$(f, \bg) \in B^{s+1}_{q,1}(\HS)\times C^\infty_0(\HS)^N$  below.
Thus, by Theorem \ref{thm:3} we know that there exists $\bv_i 
\in B^{s+2}_{q,1}(\HS)$ ($i=1,2$) such that $\bv= \bv_1 + \bv_2$ and there hold
\begin{align}
\|(\lambda, \bar\nabla^2)\bv_1\|_{B^{s}_{q,1}} &\leq C|\lambda|^{-\frac{\sigma}{2}}
\|\bg\|_{B^{s+\sigma}_{q,1}(\HS)},\label{6.21.1} \\
\|(\lambda, \bar\nabla^2)\bv_1\|_{B^{s}_{q,1}} &\leq C|\lambda|^{-(1-\frac{\sigma}{2})}
\|\bg\|_{B^{s-\sigma}_{q,1}(\HS)},\label{6.21.2} \\
\|(\theta, \bv_2)\|_{B^{s+1}_{q,1}(\HS) \times B^{s+2}_{q,1}(\HS)}
&\leq C|\lambda|^{-1}\|(f, \bg)\|_{\CH^s_{q,1}(\HS)},\label{6.21.3}\\
\|(\pd_\lambda \theta, \pd_\lambda \bv_2)\|_{B^{s+1}_{q,1}(\HS) \times B^{s+2}_{q,1}(\HS)}
&\leq C|\lambda|^{-2}\|(f, \bg)\|_{\CH^s_{q,1}(\HS)} \label{6.21.4}
\end{align}
for every $\lambda \in \Sigma_\mu +\gamma$. 
Here, $\gamma$ and $C$ depend on $\rho_*$ and 
$\|\tilde \eta_0\|_{B^{s+1}_{q,1}(\HS)}$.
\par
Let $\Gamma = \Gamma_+ \cup \Gamma_-$ be a contour in the complex plane $\BC$ defined by
$$\Gamma_\pm = \{\lambda = re^{i(\pi\pm\epsilon)} \mid r \in (0, \infty)\}.$$
Here, $\epsilon \in (0, \pi/2)$.  According to a  well-known  Holomorphic 
semigroup theory  (cf. \cite[p.257]{YK}), $T(t)$ is represented by
$$T(t)(f, \bg) = \frac{1}{2\pi i}\int_{\Gamma + \gamma} e^{\lambda t}
(\lambda\bI+\CA)^{-1}(f, \bg) \,\d\lambda \quad\text{for $t>0$}.
$$
Noticing that $(\lambda\bI+\CA)^{-1}(f, \bg) = (\theta, \bv)$, we have 
\begin{align*}
T_1(t)(f, \bg)& =\frac{1}{2\pi i}\int_{\Gamma + \gamma} e^{\lambda t}
\theta \,\d\lambda, \\
T_2(t)(f, \bg) & = \frac{1}{2\pi i}\int_{\Gamma + \gamma} e^{\lambda t}
\bv \,\d\lambda = \sum_{i=1}^2 \frac{1}{2\pi i}\int_{\Gamma + \gamma} e^{\lambda t}
\bv_i \,\d\lambda.
\end{align*}
Set 
$$T_{21}(t)\bg = \frac{1}{2\pi i}\int_{\Gamma + \gamma} e^{\lambda t}
\bv_1\,\d\lambda, \quad T_{22}(t)(f, \bg) = 
\frac{1}{2\pi i}\int_{\Gamma + \gamma} e^{\lambda t}
\bv_2\,\d\lambda.
$$
We first estimate $T_{21}(t)\bg$. 
By change of variable: $\lambda t= \ell$ and  by \eqref{6.21.1} and \eqref{6.21.2}, we have
\begin{equation}\label{6.21.3*}\begin{aligned}
\|\bar\nabla^2 T_{21}(t)\bg\|_{B^s_{q,1}(\HS)}
&\leq Ce^{\gamma t}t^{-1+\frac{\sigma}{2}}\bg\|_{B^{s+\sigma}_{q,1}(\HS)}, \\
\|\bar\nabla^2 T_{21}(t)(f, \bg)\|_{B^s_{q,1}(\HS)}
&\leq Ce^{\gamma t}t^{-1-\frac{\sigma}{2}}\|\bg\|_{B^{s-\sigma}_{q,1}(\HS)}.
\end{aligned}\end{equation}
In fact, noting that ${\rm Re}\, e^{\lambda t} = e^{t(\gamma + r\cos(\pi\pm\epsilon)}
= e^{\gamma t}e^{-tr\cos\epsilon}$ for $\lambda 
\in \Gamma_\pm + \gamma$, by \eqref{6.21.1} we have
\begin{align*}
\|\bar\nabla^2T_{21}(t)\bg\|_{B^s_{q,1}(\HS)}
& \leq Ce^{\gamma t}\int^\infty_0e^{-tr\cos\epsilon}
\|\bar\nabla^2\bv_1\|_{B^s_{q,1}(\HS)}\, \d r \\ 
& \leq Ce^{\gamma t}\int^\infty_0e^{-tr\cos\epsilon}r^{-\frac{\sigma}{2}}
\, \d r \, \|\bg\|_{B^{s+\sigma}_{q,1}(\HS)}
\\
& = Ce^{\gamma t}t^{-1+\frac{\sigma}{2}}
\int^\infty_0e^{-s\cos\epsilon}s^{-\frac{\sigma}{2}}
\, \d s \, \|\bg\|_{B^{s+\sigma}_{q,1}(\HS)}.
\end{align*}
Thus, we have the first inequality in \eqref{6.21.3*}. 
To prove the second inequality in \eqref{6.21.3*}, we write
$$\bar\nabla^2T_{21}(t)\bg
= -\frac{1}{2\pi i t}\int_{\Gamma+\gamma}
e^{\lambda t}\pd_\lambda (\bar\nabla^2\bv_1)\,d\lambda.
$$
And then, by \eqref{6.21.2}
\begin{align*}
\|\bar\nabla^2T_{21}(t)\bg\|_{B^s_{q,1}(\HS)}
& \leq Ct^{-1}e^{\gamma t}\int^\infty_0e^{-tr\cos\epsilon}
\|\bar\nabla^2\pd_\lambda \bv_1\|_{B^s_{q,1}(\HS)}\, \d r \\ 
& \leq Ct^{-1}e^{\gamma t}\int^\infty_0e^{-tr\cos\epsilon}r^{-1+\frac{\sigma}{2}}
\, \d r \, \|\bg\|_{\CH_{-\sigma}}
\\
& = Ce^{\gamma t}t^{-1-\frac{\sigma}{2}}
\int^\infty_0e^{-s\cos\epsilon}s^{-1+\frac{\sigma}{2}}
\, \d s \, \|(f, \bg)\|_{\CH_{-\sigma}}. 
\end{align*}
Thus, we have the second inequality of \eqref{6.21.3*}.
\par
Noting that $|\lambda| \geq \gamma \sin \epsilon$ for $\lambda \in \Sigma_{\mu} + \gamma$, 
by \eqref{6.21.3} and \eqref{6.21.4} we have 
\begin{align*}
\|(\theta, \bv_2)\|_{B^{s+1}_{q,1}(\HS) \times B^{s+2}_{q,1}(\HS)}
&\leq C|\lambda|^{-1}\|(f, \bg)\|_{\CH^s_{q,1}(\HS)}\\
&\leq C(\gamma\sin\epsilon)^{-(1-\frac{\sigma}{2})}|\lambda|^{-\frac{\sigma}{2}}
\|(f, \bg)\|_{\CH^s_{q,1}(\HS)}\\
\|(\pd_\lambda \theta, \pd_\lambda \bv_2)\|_{B^{s+1}_{q,1}(\HS) \times B^{s+2}_{q,1}(\HS)}
&\leq C|\lambda|^{-2}\|(f, \bg)\|_{\CH^s_{q,1}(\HS)} \\
&\leq C(\gamma\sin\epsilon)^{-(1+\frac{\sigma}{2})}
|\lambda|^{-(1-\frac{\sigma}{2})}\|(f, \bg)\|_{\CH^s_{q,1}(\HS)} 
\end{align*}
for $\lambda \in \Sigma_\mu + \gamma$.
Employing the same arguments as in the proof of \eqref{6.21.3*}, we have
\begin{equation}\begin{aligned}\label{6.21.4*}
\|T_1(t)(f, \bg)\|_{B^{s+1}_{q,1}(\HS)} &\leq Ce^{\gamma t}t^{-1+\frac{\sigma}{2}}
\|(f, \bg)\|_{\CH^s_{q,1}(\HS)},  \\
\|T_1(t)(f, \bg)\|_{B^{s+1}_{q,1}(\HS)} &\leq Ce^{\gamma t}t^{-1-\frac{\sigma}{2}}
\|(f, \bg)\|_{\CH^s_{q,1}(\HS)}, \\
\|T_{22}(t)(f, \bg)\|_{B^{s+2}_{q,1}(\HS)} &\leq Ce^{\gamma t}t^{-1+\frac{\sigma}{2}}
\|(f, \bg)\|_{\CH^s_{q,1}(\HS)},  \\
\|T_{22}(t)(f, \bg)\|_{B^{s+2}_{q,1}(\HS)} &\leq Ce^{\gamma t}t^{-1-\frac{\sigma}{2}}
\|(f, \bg)\|_{\CH^s_{q,1}(\HS)}, 
\end{aligned}\end{equation}
From \eqref{6.21.3*} and \eqref{6.21.4*}, we have
\begin{equation}\label{6.21.5}\begin{aligned}
\int^\infty_0 e^{-\gamma t}\|T_{21}(t)\bg\|_{B^{s+2}_{q,1}(\HS)}\,\d t
\leq C\|\bg\|_{B^s_{q,1}(\HS)}, \\
\int^\infty_0 e^{-\gamma t}\|T_{1}(t)(f, \bg)\|_{B^{s+1}_{q,1}(\HS)}\,\d t
\leq C\|(f, \bg)\|_{\CH^s_{q,1}(\HS)}, \\
\int^\infty_0 e^{-\gamma t}\|T_{22}(t)(f, \bg)\|_{B^{s+2}_{q,1}(\HS)}\,\d t
\leq C\|(f, \bg)\|_{\CH^s_{q,1}(\HS)}.
\end{aligned}\end{equation}
In fact, we write 
\begin{align*}
&\int^\infty_0 e^{-\gamma t}\|\bar\nabla^2T_{21}(t)\bg\|_{B^{s+2}_{q,1}(\HS)}\,\d t \\
&\quad = \sum_{j \in \BZ} \int^{2^{(j+1)}}_{2^j}e^{-\gamma t}
\|\bar\nabla^2T_{21}(t)\bg\|_{B^{s+2}_{q,1}(\HS)}\, \d t\\
& \quad\leq \sum_{j \in \BZ} \int^{2^{(j+1)}}_{2^j}
\sup_{t \in (2^j, 2^{j+1})}(e^{-\gamma t}
\|\bar\nabla^2T_{21}(t)\bg\|_{B^{s+2}_{q,1}(\HS)})\,\d t
\\
& \quad =\sum_{j \in \BZ} 2^j
\sup_{t \in (2^j, 2^{j+1})}(e^{-\gamma t}\|\bar\nabla^2T_{21}(t)\bg\|_{B^{s+2}_{q,1}(\HS)}).
\end{align*}
Setting $a_j = \sup_{t \in (2^j, 2^{j+1})}e^{-\gamma t}
\|\bar\nabla^2T_{21}(t)\bg\|_{B^{s+2}_{q,1}(\HS)}$, 
we have
$$
\int^\infty_0 e^{-\gamma t}\|\bar\nabla^2T_{21}(t)\bg\|_{B^{s+2}_{q,1}(\HS)}\,\d t
\leq 2((2^ja_j))_{\ell_1} = 2((a_j)_{j \in \BZ})_{\ell^1_1}. 
$$
Here and in the following, 
$\ell^s_q$ denotes the set of all sequences $(2^{js}a_j)_{j \in \BZ}$ such that 
\begin{align*}
\|((a_j)_{j \in \BZ})\|_{\ell^s_q} &= \Bigl\{\sum_{j \in \BZ} 
(2^{js}|a_j|)^q \Bigr\}^{1/q} < \infty \quad\text{for $1 \leq q < \infty$},  \\
\|((a_j)_{j \in \BZ})\|_{\ell^s_\infty} &= \sup_{j \in \BZ} 
2^{js}|a_j| < \infty \quad \text{for  $q = \infty$}.
\end{align*}
By \eqref{6.21.3*}, we have
$$\sup_{j \in \BZ} 2^{j(1-\frac{\sigma}{2})}a_j \leq C\|\bg\|_{B^{s+\sigma}_{q,1}(\HS)}, 
\quad\sup_{j \in \BZ} 2^{j(1+\frac{\sigma}{2})}a_j \leq C\|\bg\|_{B^{s-\sigma}_{q,1}(\HS)}
$$
Namely, we have
$$\|(a_j)\|_{\ell_\infty^{1-\frac{\sigma}{2}}} \leq C\|\bg\|_{B^{s+\sigma}_{q,1}(\HS)},
\quad \|(a_j)\|_{\ell_\infty^{1+ \frac{\sigma}{2}}} \leq C\|\bg\|_{B^{s-\sigma}_{q,1}(\HS)}.
$$
According to \cite[5.6.1.Theorem]{BL},  
we know that $\ell^1_1 = (\ell^{1-\frac{\sigma}{2}}_\infty,
\ell^{1+\frac{\sigma}{2}}_\infty)_{1/2, 1}$, 
where
$(\cdot, \cdot)_{\theta, q}$ denotes the real interpolation functor, and therefore we have 
\begin{equation}\label{L1est:1}
\int^\infty_0 e^{-\gamma t}\|T_{21}(t)\bg\|_{B^{s+2}_{q,1}(\HS)}\,\d t
\leq C\|\bg\|_{(B^{s+\sigma}_{q,1}(\HS), B^{s-\sigma}_{q,1}(\HS))_{1/2, 1}}
\leq C\|(f, \bg)\|_{B^{s}_{q,1}(\HS)(\HS)}.
\end{equation}
for any $\bg \in C^\infty_0(\HS)$. But, $C^\infty_0(\HS)$ is dense in
$B^s_{q,1}(\HS)$, so the estimate \eqref{L1est:1} holds for any $\bg \in B^s_{q,1}(\HS).$
\par 
Employing completely the same argument, by \eqref{6.21.4*} we have
\begin{equation}\label{L1est:2}\begin{aligned}
\int^\infty_0 e^{-\gamma t}\|T_1(t)(f, \bg)\|_{B^{s+1}_{q,1}(\HS)}\,\d t
&\leq C\|(f, \bg)\|_{\CH^s_{q,1}(\HS)}, \\
\int^\infty_0 e^{-\gamma t}\|T_{22}(t)(f, \bg)\|_{B^{s+2}_{q,1}(\HS)}\,\d t
&\leq C\|(f, \bg)\|_{\CH^s_{q,1}(\HS)}
\end{aligned}\end{equation}
for any $f \in B^{s+1}_{q,1}(\HS)$ and $\bg \in C^\infty_0(\HS)$.  
But, $C^\infty_0(\HS)$ is dense in
$B^s_{q,1}(\HS)$, so the estimate \eqref{L1est:2} holds for any $(f, \bg) \in \CH^s_{q,1}(\HS)$.
\par

By equations \eqref{semi:1} with $F=0$ and $\bG=0$, we have
\begin{align*}
\pd_t T_1(t)(f, \bg) &= -\eta_0(x)\dv T_{2}(t)(f, \bg), \\
\pd_tT_{2}(t)(f, \bg) & =(\eta_0)^{-1}( \alpha\Delta T_{2}(t)(f, \bg) 
+ \beta\nabla\dv T_{2}(t)(f, \bg)
 - \nabla(P'(\eta_0)T_1(t)(f\, \bg)).
\end{align*} 
Recalling that $T_2(t)(f, \bg) = T_{21}(t)\bg + T_{22}(t)(f, \bg)$,
by \eqref{L1est:1} and \eqref{L1est:2},  we have 
\begin{align*}
\int^\infty_0 e^{-\gamma t}
\|\pd_tT_1(t)(f,\bg)\|_{B^{s+1}_{q,1}(\HS)}\,\d t
&\leq C(\rho_*+\|\tilde\eta_0\|_{B^{s+1}_{q,1}(\HS)})
\int^\infty_0 e^{-\gamma t}
\|\dv T_2(t)(f, \bg)\|_{B^{s+1}_{q,1}(\HS)}\,\d t \\
& \leq C(\rho_*+\|\tilde\eta_0\|_{B^{s+1}_{q,1}(\HS)})
\|(f, \bg)\|_{\CH^s_{q,1}(\HS)}, \\
\int^\infty_0 e^{-\gamma t}\|\pd_t T_2(t)(f, \bg)\|_{B^{s}_{q,1}(\HS)}\,\d t
& \leq  C(\rho_*, \|\tilde\eta_0\|_{B^{N/q}_{q,1}})
\Bigl\{\int^\infty_0\|\nabla^2 T_2(t)(f, \bg)\|_{B^s_{q,1}(\HS)} \\
&\quad 
+ C(\rho_*, \|\tilde\eta_0\|_{B^{s+1}_{q,1}(\HS)})
\int^\infty_0 \|T_1(t)(f, \bg)\|_{B^{s+1}_{q,1}(\HS)}\,\d t\Bigr\} \\
& \leq C(\rho_*, \|\tilde\eta_0\|_{B^{s+1}_{q,1}(\HS)})
\|(f, \bg)\|_{\CH^s_{q,1}(\HS)}.
\end{align*}
This completes the proof of Theorem \ref{thm:t.2}. \end{proof}

\begin{cor}\label{semi.1}
Let $1 < q < \infty$, $-1 + N/q \leq s , 1/q$,  and $T > 0$. 
  Let $\eta_0(x)=\rho_* + \tilde\eta_0(x)$ be a 
function given in Theorem \ref{thm:1}.  Then, for any 
$(f, \bg) \in \CH$, $F\in L_1((0, T),  B^{s+1}_{q,1}(\HS))$ and 
$\bG \in L_1((0, T), B^s_{q,1}(\HS)^N)$, problem \eqref{semi:1} admits unique
solutions $\rho$ and $\bu$ with
$$\Pi \in W^1_1((0, T), B^{s+1}_{q,1}(\HS)), \quad
\bU \in L_1((0, T), B^{s+2}_{q,1}(\HS)^N) \cap W^1_1((0, T), B^s_{q,1}(\HS)^N).
$$
Moreover,  there exist constant constants  $\gamma>0$ and $C$ 
depending on $\rho_*$ and 
 $\|\tilde\eta_0\|_{B^{s+1}_{q,1}(\HS)}$ 
such that $\rho$ and $\bu$ satisfy the following maximal
$L_1$-$\CH^s_{q,1}(\HS)$ estimate:
\begin{align*}
\|(\pd_t, \bar\nabla^2)\bU\|_{L_1((0, T), B^s_{q,1}(\HS))}
+ \|(1, \pd_t)\Pi\|_{L_1((0, T), B^{s+1}_{q,1}(\HS))} 
\leq Ce^{\gamma T}(\|(f, \bg)\|_{\CH}
+ \|(F, \bG)\|_{L_1((0, T), \CH)}).
\end{align*}
\end{cor}
\begin{proof} 
Let $F_0$ and $\bG_0$ be zero extension of $F$ and $\bG$ outside of $(0, T)$
interval.  Using $\{T(t)\}_{t\geq 0}$, we can write
$$(\Pi, \bU)(t) = T(t)(f, \bg) + \int^t_0T(t-s)(F_0(\cdot, s), \rho_0^{-1}(\cdot)\bG_0(\cdot, s))\,ds.
$$
Let $\gamma$ and $C$ be the constant given in Theorem \ref{thm:t.2}. 
By Fubini's theorem, we have
\begin{align*}
&\int^\infty_0e^{-\gamma t}\|\bar\nabla^2
\int^t_0T_2(t-\ell)(F_0, \eta_0^{-1}\bG_0)(\ell)\,\d\ell\|_{B^s_{q,1}(\HS)}\, \d t \\
&\quad \leq \int^\infty_0\Bigl\{ \int^\infty_\ell e^{-\gamma t}
\|\bar\nabla^2T_2(t-\ell)(F_0, \eta_0^{-1}\bG_0)(\ell)\|_{B^s_{q,1}(\HS)}\,\d t\Bigr\}\,\d \ell \\
&\quad = \int^\infty_0e^{-\gamma \ell}
\Bigl\{ \int^\infty_0 e^{-\gamma t}
\|\bar\nabla^2 T_2(t)(F_0, \eta_0^{-1}\bG_0)(\ell)\|_{B^{s}_{q,1}(\HS)}\,\d t\Bigr\}\,\d \ell \\
&\quad \leq C\int^\infty_0 e^{-\gamma \ell} \|(F_0(\cdot, \ell), 
\eta_0^{-1}\bG_0(\cdot, \ell)\|_{\CH^s_{q,1}(\HS)}\,\d \ell \\
&\quad \leq C\|(F, \bG)\|_{L_1((0, T), \CH^s_{q,1}(\HS))}.
\end{align*}
Employing completely the same argument, we have
$$\int^\infty_0e^{-\gamma t}\|
\int^t_0T_1(t-\ell)(F_0, \rho^{-1}_0\bG_0)(\ell)\,\d \ell\|_{B^{s+1}_{q,1}(\HS)}\, \d t 
\leq C\|(F, \bG)\|_{L_1((0, T), \CH^s_{q,1}(\HS))}.
$$
Therefore, we have
$$\int^\infty_0 e^{-\gamma t}(\|\rho(\cdot, t)\|_{B^{s+1}_{q,1}(\HS)}
+ \|\bu(\cdot, t)\|_{B^{s+2}_{q,1}(\HS)})\,dt
\leq C(\|(f, \bg)\|_{\CH} + \|(F, \bG)\|_{L_1((0, T), \CH)}), 
$$
which implies that
$$e^{-\gamma T}\int^T_0(\|\rho(\cdot, t)\|_{B^{s+1}_{q,1}(\HS)}
+ \|\bu(\cdot, t)\|_{B^{s+2}_{q,1}(\HS)})\,dt
\leq C(\|(f, \bg)\|_{\CH} + \|(F,  \bG)\|_{L_1((0, T), \CH)}).
$$
Therefore, we have
$$\int^T_0(\|\rho(\cdot, t)\|_{B^{s+1}_{q,1}(\HS)}
+ \|\bu(\cdot, t)\|_{B^{s+2}_{q,1}(\HS)})\,dt
\leq Ce^{\gamma T}(\|(f, \bg)\|_{\CH} + \|(F, \bG)\|_{L_1((0, T), \CH)}).
$$
Here, $\gamma$ and $C$ are  constants depending on $\rho_*$, 
$\|\tilde\eta_0\|_{B^{s+1}_{q,1}(\HS)}$. \par 
To show the estimate of time derivatives, we use the relations:
\begin{align*}
\pd_t\Pi& = -\eta_0\dv\bu + F, \\
\pd_t\bU& =(\eta_0)^{-1}(\alpha\Delta\bU + \beta\nabla\dv\bU - 
\nabla(P(\eta_0)\Pi)+ \bG),
\end{align*}
and then, 
\begin{align*}
&\int^T_0(\|\pd_t\Pi(\cdot, t)\|_{B^{s+1}_{q,1}(\HS)}
+ \|\pd_t\bU(\cdot, t)\|_{B^{s}_{q,1}(\HS)})\,dt \\
&\leq C(\rho_*, \|\tilde\eta_0\|_{B^{s+1}_{q,1}})(
\int^T_0(\|\Pi(\cdot, t)\|_{B^{s+1}_{q,1}(\HS)}
+ \|\bU(\cdot, t)\|_{B^{s+2}_{q,1}(\HS)})\,dt
+ \|(F, \bG)\|_{L_1((0, T), \CH^s_{q,1}(\HS))}) \\
& \leq C(\rho_*, \|\tilde\eta_0\|_{B^{s+1}_{q,1}})e^{\gamma T}
(\|(f, \bg)\|_{\CH^s_{q,1}(\HS)} + \|(F, \bG)\|_{L_1((0, T), \CH^s_{q,1}(\HS))}).
\end{align*}
This completes the proof of  Corollary \ref{semi.1}.

\end{proof}


\section{A proof of Theorem \ref{thm:2}}

In this section, we shall prove Theorem \ref{thm:2}. Let $\eta_0(x) = \rho_* + \tilde\eta_0(x)$
with $\tilde\eta_0(x) \in B^{s+1}_{q,1}(\HS)$ and assume that $\eta_0(x)$ satisfy the assumption
\eqref{assump:0}.  Let $\rho_0(x) = \rho_* + \tilde\rho_0(x)$ with $\tilde\rho_0(x) \in 
B^{s+1}_{q,1}(\HS)$.  Let $\omega>0$ be a small number determined late and assume that 
\begin{equation}\label{assump:1}
\|\tilde\rho_0-\tilde\eta_0\|_{B^{s+1}_{q,1}(\HS)} < \omega
\end{equation}
Let $\bu_0 \in B^s_{q,1}(\HS)^N$. We consider equations \eqref{ns:2}. 
By setting $\rho = \rho_0 + \theta$  we write equations \eqref{ns:2} as follows: 
\begin{equation}\label{appro:ns.1}\left\{\begin{aligned}
\pd_t\theta+\eta_0\dv\bu = (\eta_0-\rho_0-\theta)\dv\bu 
 + F(\theta+\rho_0, \bu)& 
&\quad&\text{in $\HS\times(0, T)$}, \\
\eta_0 \pd_t\bu - \alpha\Delta \bu  -\beta\nabla\dv\bu 
+  \nabla (P'(\eta_0)\theta)
 = -\nabla P(\rho_0) + \bG(\theta+\rho_0, \bu) + \tilde\bG(\theta, \bu) && \quad&\text{in $\HS\times(0, T)$}, \\
\bu|_{\pd\HS} =0, \quad (\theta, \bu)|_{t=0} = (0, \bu_0)&
& \quad&\text{in $\HS$},
\end{aligned}\right.\end{equation}
where we have set $\tilde\bG(\theta, \bu) 
 = (\eta_0 - \rho_0-\theta)\pd_t\bu  -  \nabla( P(\rho_0+\theta)-P(\rho_0)
-P'(\eta_0)\theta)$. 

To prove Theorem \ref{thm:2}, we use  the Banach contraction
mapping principle. To this end, we introduce an energy functional
$E_T$ and an underlying space $S_{T, \omega}$ defined by 
\begin{align*}
E_T&(\eta, \bw)  = \|(\eta, \pd_t\eta)\|_{L_1((0, T), B^{s+1}_{q,1}(\HS))}
+ \|\bw\|_{L_1((0, T), B^{s+2}_{q,1}(\HS))}
+ \|\pd_t\bw\|_{L_1((0, T), B^s_{q,1}(\HS))}, \\[0.5pc]
S_{T, \omega}  &= \left\{ 
(\eta, \bw) \, \left | 
\begin{aligned} \eta &\in W^1_1((0, T), B^{s+1}_{q,1}(\HS)), \enskip 
\bw \in L_1((0, T), B^{s+2}_{q,1}(\HS)^N)
\cap W^1_1((0, T), B^s_{q,1}(\HS)^N) \\
(\eta, &\bw)|_{t=0} = (0, \bu_0), \quad E_T(\eta, \bw) \leq \omega,
\quad \int^T_0\|\nabla \bw(\cdot, \tau)\|_{B^{N/q}_{q,1}(\HS)}\,\d\tau \leq c_1
\end{aligned}
\right.\right\}.
\end{align*}
Here, $T>0$, $\omega>0$ and $c_1>0$  are small constants 
chosen later.  In particular, $c_1$ is chosen in such a way that
$$\Bigl\|\int^T_0\nabla\bw(\cdot, \tau)\,\d\tau\Bigr\|_{L_\infty(\HS)}
\leq C\int^T_0\|\nabla \bw(\cdot, \tau)\|_{B^{N/q}_{q,1}(\HS)} \leq
Cc_1 \leq c_0,
$$
where $c_0$ is a constant appearing in \eqref{assump:2}.  Thus, the constant 
$c_1$ guarantees that the Lagrange map $y=X_\bw(x, t)$ is $C^1$ diffeomorphism
from $\Omega$ onto $\Omega$.
\par
Given $(\theta, \bu) \in S_{T, \omega}$, let $\eta$ and $\bw$ be solutions
to the system of linear equations:
\begin{equation}\label{st:2}\left\{\begin{aligned}
\pd_t\eta+\eta_0\dv\bw =
(\eta_0 - \rho_0-\theta)\dv\bu + F(\rho_0+\theta, \bu)& 
&\quad&\text{in $\HS\times(0, T)$}, \\
\eta_0\pd_t\bw - \alpha\Delta \bw  -\beta\nabla\dv\bw
+  \nabla(P'(\eta_0) \eta) = -\nabla P(\rho_0) +\bG(\rho_0+\theta, \bu)
+\tilde\bG(\theta, \bu)
& &\quad&\text{in $\HS\times(0, T)$}, \\
\bw|_{\pd\HS} =0, \quad (\eta, \bw)|_{t=0} = (0, \bu_0)
& &\quad&\text{in $\HS$}.
\end{aligned}\right.\end{equation}
Let $\eta_\ba$ and $\bw_\ba$ be solutions of the system of linear equations:
\begin{equation}\label{st:3}\left\{\begin{aligned}
\pd_t\eta_\ba+\eta_0\dv\bw_\ba =0&
&\quad&\text{in $\HS\times(0, \infty)$}, \\
\eta_0\pd_t\bw_\ba - \alpha\Delta \bw_\ba   -\beta\nabla\dv\bw_\ba
+  \nabla(P'(\eta_0) \eta_\ba) = -\nabla P(\rho_0)
& &\quad&\text{in $\HS\times(0, \infty)$}, \\
\bw_\ba|_{\pd\HS} =0, \quad (\eta_\ba, \bw_\ba)|_{t=0} = (0, \bu_0)
& &\quad&\text{in $\HS$}.
\end{aligned}\right.\end{equation}
We will choose $T>0$ small enough later, and so for a while we assume that
$0 < T < 1$. 
By Corollary \ref{semi.1}, we know the unique existence 
of solutions $\eta_\ba$ and $\bw_\ba$ satisfying the regularity conditions:
$$\eta_\ba \in W^1_1((0, 1), B^{s+1}_{q,1}(\HS)), 
\quad 
\bw_\ba \in L_1((0, 1), B^{s+2}_{q,1}(\HS)^N)
\cap W^1_1((0, 1), B^s_{q,1}(\HS)^N)
$$ 
as well as the estimates:
\begin{equation}\label{est:2}\begin{aligned}
&\|(\eta_\ba, \pd_t\eta_\ba)\|_{L_1((0, 1), B^{s+1}_{q,1}(\HS))}
+ \|(\pd_t, \bar\nabla^2)\bw_\ba\|_{L_1((0, 1), B^s_{q,1}(\HS))} 
\\
&\quad \leq Ce^{\gamma}(\|\bu_0\|_{B^s_{q,1}(\HS)} + \|\nabla P(\rho_0)\|_{B^s_{q,1}(\HS)}).
\end{aligned}\end{equation}
Here, $\gamma$ is a constant depending on $\rho_*$, 
 $\|\tilde\eta_0\|_{B^{s+1}_{q,1}(\HS)}$  
given in Corollary \ref{semi.1}. Here and in the following, $C$ denotes a  general constant
depending at most on $\rho_*$ and $\|\tilde\eta_0\|_{B^{s+1}_{q,1}(\HS)}$, 
which is changed from line to line,  
but independent of $\omega$ and $T$.   \par
In view of \eqref{est:2},  $\eta_\ba$ and $\bw_\ba$ satisfy $E_1(\eta_\ba, \bw_\ba) < \infty$, 
 and so we choose $T \in (0, 1)$ small enough in such a way that
\begin{equation}\label{est:5}
E_T(\eta_\ba, \bw_\ba) \leq \omega/2.
\end{equation}
Let $\rho$ and $\bv$ be solutions to the system of linear equations:
\begin{equation}\label{st:40}\left\{\begin{aligned}
\pd_t\rho+\eta_0\dv\bv =(\eta_0-\rho_0 -\theta))\dv\bu +
F(\theta+\rho_0, \bu)& 
&\quad&\text{in $\HS\times(0, T)$}, \\
\eta_0\pd_t\bv- \alpha\Delta \bv-\beta\nabla\dv\bv
+  \nabla(P'(\eta_0) \rho) = \bG(\theta+\rho_0, \bu) + \tilde\bG(\theta, \bu)
& &\quad&\text{in $\HS\times(0, T)$}, \\
\bv|_{\pd\HS} =0, \quad (\rho, \bv)|_{t=0} = (0, 0)
& &\quad&\text{in $\HS$}.
\end{aligned}\right.\end{equation}
Applying Corollary \ref{semi.1}, 
we see the existence of solutions $\rho$ and $\bv$ of equations
\eqref{st:40} satisfying the regularity condition:
$$\rho \in W^1_1((0, T), B^{s+1}_{q,1}(\HS)), 
\quad
\bv \in L_1((0, T), B^{s+2}_{q,1}(\HS)^N) \cap W^1_1((0, T), B^{s}_{q,1}(\HS)^N)
$$
as well as the estimate:
\begin{equation}\label{est:3}\begin{aligned}
E_T(\rho, \bv) \leq Ce^{\gamma T}(&\|(\eta_0-\rho_0 - \theta)\dv\bu, 
F(\theta+\rho_0, \bu)\|_{L_1((0, T), B^{s+1}_{q,1}(\HS))} \\
&+ \|(\bG(\theta+\rho_0, \bu), \tilde\bG(\theta, \bu)) 
\|_{L_1((0, T), B^s_{q,1}(\HS))}).
\end{aligned}\end{equation}
Here, we notice that $\gamma$ and  $C$ depend on $\rho_*$ and $\|\tilde\eta_0\|_{B^{s+1}_{q,1}(\HS)}$
but is independent of $\omega$ and $T$.
\par

Now, we  shall show that there exist  constants $C>0$ and $\omega >0$ such that 
\begin{equation}\label{est:4}\begin{aligned}
\|(\eta_0-\rho_0-\theta)\dv\bu,
F(\theta+\rho_0, \bu)\|_{L_1((0, T), B^{s+1}_{q,1}(\HS))}
&+ \|(\bG(\theta+\rho_0, \bu), \tilde\bG(\theta, \bu)) 
\|_{L_1((0, T), B^s_{q,1}(\HS))}  \\
&\leq C(\omega^2 + \omega^3).
\end{aligned}\end{equation}
If we show \eqref{est:4}, then by \eqref{est:3} we have
\begin{equation}\label{est:7}
E_T(\rho, \bv) \leq Ce^{\gamma T}(\omega^2 + \omega^3).
\end{equation}
Choose  $\omega>0$ and $T>0$  so small that $Ce(\omega + \omega^2) \leq 1/2$ 
and  $\gamma T\leq 1$.  Then, 
we have
\begin{equation}\label{est:6}
E_T(\rho, \bv) < \omega/2,
\end{equation}
which, combined with \eqref{est:5}, implies that 
$\eta = \eta_\ba +  \rho$ and $\bw =\bw_\ba + \bv$ satisfy
equations \eqref{st:2} and $E_T(\eta, \bw) < \omega$.
Especially, $\omega$ is chosen so small that 
$$\int^T_0\|\nabla\bw(\cdot, \tau)\|_{B^{N/q}_{q,1}(\HS)}\,\d\tau 
\leq CE_T(\eta, \bw) \leq C\omega \leq c_1.$$
As a consequence, $(\eta, \bw) \in S_{T, \omega}$. 
Thus,  if we define the map $\Phi$ by $\Phi(\theta, \bu) = (\eta, \bw)$,
then  $\Phi$ maps $S_{T, \omega}$ into $S_{T, \omega}$. \par

Now, we shall show \eqref{est:4}.  For notational simplicity, we omit $\HS$ below. 
Notice that $B^{N/q}_{q,1}$ is a Banach algebra (cf. \cite[Proposition 2.3]{H11}). 
By Lemma \ref{lem:APH} and 
the assumption: $N/q\leq  s+1$, we see that $B^{s+1}_{q,1}$ is also a Banach
algebra.  In fact, 
$$\|uv\|_{B^{s+1}_{q,1}} \leq \|(\nabla u)v\|_{B^s_{q,1}} + \|u\bar\nabla v\|_{B^s_{q,1}}
\leq C(\|\nabla u\|_{B^s_{q,1}}\|v\|_{B^{N/q}_{q,1}} +
\|u\|_{B^{N/q}_{q,1}}\|\bar\nabla v\|_{B^s_{q,1}})
\leq C\|u\|_{B^{s+1}_{q,1}}\|v\|_{B^{s+1}_{q,1}}.
$$
\par
We first estimate $(\eta_0-\rho_0-\theta)\dv\bu$ and 
$F(\theta+\rho_0, \bu)$. 
By Lemma \ref{lem:APH} and \eqref{assump:1},  we have
\begin{equation}\label{nonest:1}
\|(\eta_0-\rho_0)\dv\bu\|_{B^{s+1}_{q,1}} \leq C\omega\|\bu\|_{B^{s+2}_{q,1}}.
\end{equation}
Since $B^{s+1}_{q,1}$ is a Banach algebral, we have
$$\|\theta\dv\bu\|_{B^{s+1}_{q,1}} \leq C\|\theta\|_{B^{s+1}_{q,1}}\|\dv\bu\|_{B^{s+1}_{q,1}}.$$
Since $\theta|_{t=0} = 0$, here and  in the sequel we use the following estimate:
\begin{equation}\label{theta:1}
\|\theta(\cdot, t)\|_{B^{s+1}_{q,1}} = \Bigl\|\int^t_0\pd_s\theta(\cdot, s)\,\d s\Bigr\|_{B^{s+1}_{q,1}(\HS)}
\leq \|\pd_t\theta\|_{L_1((0, T), B^{s+1}_{q,1}(\HS)}. 
\end{equation}
Thus, we have
$$\|\theta\dv\bu\|_{L_1((0, T), B^{s+1}_{q,1})} \leq C \|\pd_t\theta\|_{L_1((0, T), B^{s+1}_{q,1})}
\|\bu\|_{L_1((0, T), B^{s+2}_{q,1})}.
$$
We next estimate $F(\rho_0+\theta, \bu) 
= (\rho_0+\theta)((\BI - \BA_\bu):\nabla\bu)$. 
Recall that $\bu$ satisfies
\begin{equation}\label{defin:1}
\int^T_0\|\nabla\bu(\cdot, \tau)\|_{B^{N/\beta}_{q,1}}\, \d\tau \leq c_1.
\end{equation}
Since 
$B^{N/q}_{q,1} \subset L_\infty$, we have
\begin{equation}\label{defin:2}
\sup_{t \in (0, T)}\,\Bigl\|\int^t_0 \nabla \bu(\cdot, \tau)\,\d\tau\Bigr\|_{L_\infty}
\leq C\int^T_0\|\nabla \bu(\cdot, \tau)\|_{B^{N/\beta}_{q,1}}\,\d\tau 
\leq Cc_1.
\end{equation}
Choosing $c_1$ so small that $Cc_1 < 1$. 
Let $F(\ell)$ be a $C^\infty$ function defined on $|\ell| \leq Cc_1$ and 
$F(0)=0$, and 
$\BI-\BA_\bu = F(\int^t_0\nabla\bu\,\d\ell)$. In fact, 
$F(\ell) = -\sum_{j=1}^\infty \ell^j$.
Then, by Lemma \ref{lem:Hasp} and \eqref{defin:2}, we have
\begin{equation}\label{nonfun:1}
\sup_{t \in (0, T)} \|F(\int^t_0\nabla\bu\,\d\tau)\|_{B^{s+1}_{q,1}}
\leq C\int^T_0\|\nabla\bu(\cdot, \tau)\|_{B^{s+1}_{q,1}}\,d\tau.
\end{equation}
  Since $B^{s+1}_{q,1}$ is a Banach algebra, using \eqref{nonfun:1} we have
\begin{align*}
\|F(\rho_0+\theta, \bu)\|_{B^{s+1}_{q,1}}
\leq C(\|\rho_0\|_{B^{s+1}_{q,1}}+\|\theta(\cdot, t)\|_{B^{s+1}_{q,1}})
\|\bu\|_{L_1((0, T), B^{s+2}_{q,1})}\|\nabla\bu(\cdot, t)\|_{B^{s+1}_{q,1}}.
\end{align*}
Using \eqref{theta:1}, we have
$$\|F(\rho_0+\theta, \bu)\|_{L_1((0, T), B^{s+1}_{q,1})}
\leq C(\|\rho_0\|_{B^{s+1}_{q,1}}+\|\theta_t\|_{L_1((0, T), B^{s+1}_{q,1})})
\|\bu\|_{L_1((0, T), B^{s+2}_{q,1})}^2.
$$
Summing up, we have proved that 
\begin{equation}\label{mainest:1}\begin{aligned}
&\|(\eta_0-\rho_0-\theta)\dv\bu, 
F(\theta+\rho_0) \bu)\|_{L_1((0, T), B^{s+1}_{q,1}(\HS))} \\
&\quad \leq C\{\omega \|\bu\|_{L_1((0, T), B^{s+2}_{q,1})}
+ \|\pd_t\theta\|_{L_1((0, T), B^{s+1}_{q,1})}\|\bu\|_{L_1((0, T), B^{s+2}_{q,1})} \\
&\quad +(\|\eta_0\|_{B^{s+1}_{q,1}}+1)\|\bu\|_{L_1((0, T), B^{s+2}_{q,1})}^2
+ \|\pd_t\theta\|_{L_1((0, T), B^{s+1}_{q,1})}\|\bu\|_{L_1((0, T), B^{s+2}_{q,1})}^2\}.
\end{aligned}\end{equation}
Here and in the following, we use the estimate: 
$$\|\rho_0\|_{B^{s+1}_{q,1}} \leq \|\rho_0-\eta_0\|_{B^{s+1}_{q,1}} + \|\eta_0\|_{B^{s+1}_{q,1}})
\leq 1+ \|\eta_0\|_{B^{s+1}_{q,1}}.
$$

Next, we estimate $\|(\bG(\theta+\rho_0, \bu), \tilde\bG(\theta, \bu)) 
\|_{L_1((0, T), B^s_{q,1}(\HS))}$. 
By Lemma \ref{lem:APH}, the assumption: $N/d \leq s+1$,  \eqref{theta:1}, and 
\eqref{nonfun:1}, 
we have
\begin{align*}
\|(\BI-\BA_\bu)(\rho_0+\theta)\pd_t\bu\|_{B^s_{q,1}}
&\leq C\|\BI-\BA_\bu\|_{B^{N/q}_{q,1}}\|\rho_0+\theta\|_{B^{N/q}_{q,1}}
\|\pd_t\bu\|_{B^s_{q,1}} \\
& \leq C\|\bu\|_{L_1((0, T), B^{s+2}_{q,1})}(\|\rho_0\|_{B^{s+1}_{q,1}}
+ \|\pd_t\theta\|_{L_1((0, T), B^{s+1}_{q,1}})\|\pd_t\bu\|_{B^s_{q,1}}, \\
\|(\BA_u^{-1}-\BI)\dv(\BA_\bu\BA_\bu^\top:\nabla\bu)\|_{B^s_{q,1}}
&\leq C\|\BA_u^{-1}-\BI\|_{B^{N/q}_{q,1}}(\|\dv \nabla\bu\|_{B^{s}_{q,1}}
+ \|(\BA_\bu\BA_\bu^\top-\BI):\nabla\bu\|_{B^{s+1}_{q,1}})\\
& \leq C\|\bu\|_{L_1((0, T), B^{s+2}_{q,1})}(1
+ \|\bu\|_{L_1((0, T), B^{s+2}_{q,1})})\|\bu\|_{B^{s+2}_{q,1}}.
\end{align*}
Therefore, we have 
\begin{equation}\label{mainest:2}\begin{aligned}
&\|(\bG(\theta+\rho_0, \bu)\|_{L_1((0, T), B^s_{q,1})} \\
&\quad \leq 
C( \|\bu\|_{L_1((0, T), B^{s+2}_{q,1})}(\|\rho_0\|_{B^{s+1}_{q,1}}
+ \|\pd_t\theta\|_{L_1((0, T), B^{s+1}_{q,1}})\|\pd_t\bu\|_{L_1((0, T), B^s_{q,1})}\\
&\quad +
\|\bu\|_{L_1((0, T), B^{s+2}_{q,1})}(1
+ \|\bu\|_{L_1((0, T), B^{s+2}_{q,1})})\|\bu\|_{L_1((0, T), B^{s+2}_{q,1}}).
\end{aligned}\end{equation}

Next, we shall estimate 
$\tilde\bG(\theta, \bu) 
 = (\eta_0 - \rho_0)\pd_t\bu  + \nabla( P(\rho_0+\theta)-P(\rho_0)-P'(\eta_0)\theta)$. 
Using Lemma \ref{lem:APH}, \eqref{assump:1} and $N/q \leq s+1$,   we have
\begin{align*}
\|(\eta_0 - \rho_0)\pd_t\bu\|_{B^s_{q,1}}
&\leq C\|\tilde\eta_0 -\tilde \rho_0\|_{B^{N/q}_{q,1}}
\|\pd_t\bu\|_{B^s_{q,1}} \leq C\omega\|\pd_t\bu\|_{B^s_{q,1}}.
\end{align*}
To estimate the second term, we write 
\begin{align*}
&P(\rho_0+\theta)  - P(\rho_0) -P'(\eta_0)\theta \\
&\quad = P(\rho_0+\theta)-P(\rho_0) -P'(\rho_0)\theta 
+ (P'(\rho_0)-P'(\eta_0))\theta\\
&\quad =  \int^1_0(1-\ell)P''(\rho_0+\ell\theta)\,\d\ell\theta^2 
+ \int^1_0P''(\eta_0 + \ell(\rho_0-\eta_0))\,\d\ell (\rho_0-\eta_0)\theta.
\end{align*}
Write $\rho_0+\ell\theta = \eta_0 + \rho_0-\eta_0 + \ell\theta$. By \eqref{assump:1},  \eqref{theta:1}
and $E_T(\theta, \bu) < \omega$,  we see that
$$\|\rho_0-\eta_0 + \ell\theta\|_{L_\infty} 
\leq C\|\rho_0-\eta_0\|_{B^{s+1}_{q,1}} + \ell\|\pd_t\theta\|_{L_1((0, T), B^{s+1}_{q,1})}
\leq C\omega
$$
for $\ell \in (0, 1)$. 
In view of \eqref{assump:0},  if we choose $\omega > 0$ so small that 
$$\|\rho_0-\eta_0 + \ell\theta\|_{L_\infty(\HS)} < \min(\rho_1/2, \rho_2)
\quad\text{for any $\ell \in [0, 1]$}, $$
we have 
$$
\rho_1/2  < \eta_0 + \rho_0-\eta_0 +\ell\theta < 2 \rho_2
$$
for any $\ell \in [0, 1]$. Recalling that $\eta_0 = \rho_* + \tilde\eta_0$, we have
\begin{equation}\label{range:1}
\rho_1/2-\rho_* <  \tilde\eta_0 + \rho_0-\eta_0+\ell\theta \leq 2\rho_2 -\rho_*
\end{equation}
for any $\ell \in (0, 1)$.  
From this  observation, we write
\begin{align*}
&\int^1_0(1-\ell)P''(\rho_0+\ell\theta)\,\d\ell\theta^2 \\
&= \int^1_0\int^1_0(1-\ell)P'''(\rho_* + m(\tilde \eta_0 + \rho_0-\eta_0 +\ell\theta)
(\tilde\eta_0+\rho_0-\eta_0 + \ell\theta)\,\d m\d\ell\, \theta^2
+ \frac12 P''(\rho_*)\theta^2.
\end{align*}

 And also, we write $\eta_0 + \ell(\rho_0-
\eta_0) = \eta_0 + (1-\ell)(\eta_0 -\eta_0) + \ell(\rho_0-
\eta_0)$ and  observe that 
\begin{equation}\label{domain:4.1}
\|\eta_0-\eta_0 + \ell(\rho_0-
\eta_0)\|_{L_\infty}
\leq C((1-\ell)\| \eta_0-\eta_0 \|_{B^{s+1}_{q,1}} + \ell\|\rho_0-\eta_0\|_{B^{s+1}_{q,1}})
\end{equation}
for any $\ell \in (0, 1)$, In view of \eqref{assump:1}, we choose
$\omega$ so small that 
$$\rho_1/2 < \eta_0 + \ell(\rho_0-\eta_0) < 2\rho_2$$
for any $\ell \in (0, 1)$ as follows from Assumption \eqref{assump:0}, we have
\begin{equation}\label{range:2}
\rho_1/2-\rho_* <  \tilde\eta_0 + \ell(\rho_0-\eta_0) < 2\rho_2-\rho_*
\end{equation}
for any $\ell \in (0, 1)$.  From this observation, we write 
\begin{align*}
&\int^1_0P''(\eta_0 + \ell(\rho_0-\eta_0))\,\d\ell (\rho_0-\eta_0)\theta
\\
& = \int^1_0\int^1_0 P'''(\rho_* + m(\tilde\eta_0
+  \ell(\rho_0-\eta_0))(\tilde\eta_0
+  \ell(\rho_0-\eta_0)
 \,\d\ell \d m (\rho_0-\eta_0)\theta
+ P''(\rho_*)(\rho_0-\eta_0)\theta.
\end{align*}
Therefore, by Lemmas \ref{lem:APH} and \ref{lem:Hasp}, we have
\begin{align*}
\|\nabla( P(\rho_0+\theta)-P(\rho_0) 
-P'(\eta_0)\theta)\|_{B^s_{q,1}}
\leq C(\rho_*, \|\tilde\eta_0\|_{B^{s+1}_{q,1}})( \|\theta\|_{B^{s+1}_{q,1}}^2 +\|\rho_0-\eta_0\|_{B^{s+1}_{q,1}}
\|\theta\|_{B^{s+1}_{q,1}}).
\end{align*}


Putting these estimates together and using \eqref{assump:1} and \eqref{theta:1}, we have
\begin{equation}\label{mainest:3}\begin{aligned}
\|\tilde\bG(\theta, \bu)\|_{L_1((0, T), B^s_{q,1})} 
&\leq C(\rho_*, \|\tilde\eta_0\|_{B^{s+1}_{q,1}})\{
\omega(\|\pd_t\bu\|_{L_1((0, T), B^s_{q,1})}+\|\theta\|_{L_1((0, T), B^{s+1}_{q,1})}) \\
&\quad + \|\pd_t\theta\|_{L_1((0, T), B^{s+1}_{q,1})}\|\theta\|_{L_1((0, T), B^{s+1}_{q,1})}\}.
\end{aligned}\end{equation}

Combining \eqref{mainest:1}, \eqref{mainest:2}, \eqref{mainest:3}
and recalling that $E_T(\theta, \bu) \leq \omega$, we have \eqref{est:4}.
And so,  choosing  $\omega > 0$ and $T>0$  so small that
 $Ce(\omega+\omega^2) \leq 1/2$ and  $\gamma T\leq 1$,
we have  \eqref{est:6}. 
Here, $C$and $\gamma$  depends on $\rho_*$ and $\|\tilde\eta_0\|_{B^{s+1}_{q,1}}$, 
and so the smallness of $\omega$ and $T>0$ depends on $\rho_*$ and 
$\|\tilde\eta_0\|_{B^{s+1}_{q,1}}$.  
 Therefore, we see that 
$\Phi$ maps $S_{T, \omega}$ into itself.  \par

We now prove that $\Phi$ is contractive.  To this end, 
pick up two elements $(\theta_i, \bu_i) \in S_{T, \omega}$
($i=1,2$) arbitrarily, and let $(\eta_i, \bw_i) = \Phi(\theta_i, \bu_i) \in S_{T, \omega}$ be solutions of 
equations \eqref{st:2} with $(\theta, \bu) = (\theta_i, \bu_i)$. 
Let  
\begin{align*}
\Theta &= \eta_1-\eta_2, \quad 
\bU = \bw_1-\bw_2, \\
\BF&=(\eta_0-\rho_0)\dv(\bu_1-\bu_2)
-(\theta_1\dv\bu_1-\theta_2\dv\bu_2) + F(\theta_1+\rho_0, \bu_1) - F(\theta_2+\rho_0, \bu_2), \\
\BG &= \bG(\theta_1+\rho_0, \bu_1) - \bG(\theta_2+\rho_0, \bu_2) 
+\tilde \bG(\theta_1, \bu_1) - \tilde\bG(\theta_2, \bu_2).
\end{align*}
Notice that   $\Theta$ and $\bU$ satisfy equations:
\begin{equation}\label{diff:1}\left\{\begin{aligned}
\pd_t\Theta+\eta_0\dv\bU =\BF& 
&\quad&\text{in $\HS\times(0, T)$}, \\
\pd_t\bU - \alpha\Delta \bU  -\beta\nabla\dv\bU
+  \nabla(P'(\rho_0) \Theta) = \BG
& &\quad&\text{in $\HS\times(0, T)$}, \\
\bU|_{\pd\HS} =0, \quad (\Theta, \bU)|_{t=0} = (0, 0)
& &\quad&\text{in $\HS$}.
\end{aligned}\right.\end{equation}
From \eqref{est:3}, it follows that 
\begin{equation}\label{diff:2}
E_T(\eta_1-\eta_2, \bw_1-\bw_2) \leq Ce^{\gamma T}(\|\BF\|_{L_1((0, T), B^{s+1}_{q,1})}
+ \|\BG\|_{L_1((0, T), B^{s+2}_{q,1})}).
\end{equation}
We shall prove that 
\begin{equation}\label{diff:3}
\|\BF\|_{L_1((0, T), B^{s+1}_{q,1})}
+ \|\BG\|_{L_1((0, T), B^{s+2}_{q,1})}
\leq C(\omega+\omega^2) E_T(\theta_1-\theta_2, \bu_1-\bu_2).
\end{equation}
We start with estimating $\BF$.  Recall that $B^{N/q}_{q,1}$ and 
$B^{s+1}_{q,1}$ are Banach algebra. 
By Lemma \ref{lem:APH} and \eqref{assump:1}
\begin{align*}\|(\eta_0-\rho_0)\dv(\bu_1-\bu_2)\|_{L_1((0, T), B^{s+1}_{q,1})}
&\leq C\|\eta_0-\rho_0\|_{B^{s+1}_{q,1}}\|\bu_1-\bu_2\|_{L_1((0, T), B^{s+2}_{q,1})}
\\
&\leq C\omega \|\bu_1-\bu_2\|_{L_1((0, T), B^{s+2}_{q,1})}.
\end{align*}
Writing $\theta_1\dv\bu_1-\theta_2\dv\bu_2= (\theta_1-\theta_2)\dv\bu_1+ \theta_2(\dv\bu_1-\dv\bu_2)$
and using Lemma \ref{lem:APH} and \eqref{theta:1} gives 
\begin{align*}
&\|\theta_1\dv\bu_1-\theta_2\dv\bu_2\|_{B^{s+1}_{q,1}} \leq C(\|\dv\bu_1\|_{B^{s+1}_{q,1}}
\|\theta_1-\theta_2\|_{B^{s+1}_{q,1}} + \|\theta_2\|_{B^{s+1}_{q,1}}\|\dv(\bu_1-\bu_2)\|_{B^{s+1}_{q,1}})\\
&\quad \leq C(\|\bu_1\|_{B^{s+2}_{q,1}}\|\pd_t(\theta_1-\theta_2)\|_{L_1((0,T), B^{s+1}_{q,1})}
+ \|\pd_t\theta_2\|_{L_1((0, T), B^{s+1}_{q,1})}\|\bu_1-\bu_2\|_{B^{s+2}_{q,1}}).
\end{align*}
Using $E_T(\theta_i, \bu_i) \leq \omega$ ($i=1,2$), we have 
$$
\|\theta_1\dv\bu_1-\theta_2\dv\bu_2\|_{L_1((0, T), B^{s+1}_{q,1}) }
\leq C\omega E_T(\theta_1-\theta_2, \bu_1-\bu_2).
$$ 
Write 
\begin{align*}
F(\theta_1+\rho_0, \bu_1) - F(\theta_2+\rho_0, \bu_2)
&= (\theta_1-\theta_2)((\BI - \BA_{\bu_1}):\nabla\bu_1 \\
& -(\rho_0+\theta_2)(\BA_{\bu_1} - \BA_{\bu_2}):\nabla \bu_1
+ (\rho_0+\theta_2)(\BI-\BA_{\bu_2}):\nabla(\bu_1-\bu_2).
\end{align*}
Set $\BI-\BA_{\bu} = F(\int^t_0\nabla \bu)$ and 
write
\begin{align*}
&F(\int^t_0\nabla \bu_1\,\d\ell) - F(\int^t_0\nabla\bu_2\,\d\ell) 
= \int^1_0F'(\int^1_0 (\nabla\bu_2+ m\nabla(\bu_1-\bu_2))\,\d\ell)\,\d m
\int^t_0\nabla(\bu_1-\bu_2)\,\d\ell  \\
& = \Bigl\{F'(0) + \int^1_0\int^1_0F''(n\int^1_0 (\nabla\bu_2+ m\nabla(\bu_1-\bu_2))\,\d\ell)\,\d m \d n\Bigr\}
\int^t_0\nabla(\bu_1-\bu_2)\,\d\ell.
\end{align*}
By \eqref{defin:2}, we have 
\begin{align*}
&\sup_{t \in (0, T)}\Bigl\|n\int^1_0 (\nabla\bu_2+ m\nabla(\bu_1-\bu_2))\,\d\tau\Bigr\|_{L_\infty}\\
&\leq  (1-m)\sup_{t \in (0, T)}\Bigl\|\int^1_0 \nabla\bu_2\,\d\tau\Bigr\|_{L_\infty}
+ m\sup_{t \in (0, T)} \Bigl\|\int^1_0 \nabla\bu_1\,\d\tau\Bigr\|_{L_\infty}
\leq Cc_1 
\end{align*}
by Lemmas \ref{lem:APH} and \ref{lem:Hasp},  we have
$$\|F(\int^t_0\nabla \bu_1\,\d\ell) - F(\int^t_0\nabla\bu_2\,\d\ell) \|_{B^{s+1}_{q,1}}
\leq C\|\bu_1-\bu_2\|_{L_1((0, T), B^{s+2}_{q,1})}.
$$
Thus, by Lemma \ref{lem:APH} , \eqref{theta:1} and \eqref{nonfun:1}, we have
\begin{align*}
&\|F(\theta_1+\rho_0, \bu_1) - F(\theta_2+\rho_0, \bu_2)\|_{B^{s+1}_{q,1}}\\
&\quad \leq C\{\|\theta_1-\theta_2\|_{B^{s+1}_{q,1}}\|\BI - \BA_{\bu_1}\|_{B^{s+1}_{q,1}}\|\nabla\bu_1\|_{B^{s+1}_{q,1}}
+\|\rho_0+\theta_2\|_{B^{s+1}_{q,1}}\|\BA_{\bu_1} - \BA_{\bu_2}\|_{B^{s+1}_{q,1}}\|\nabla \bu_1\|_{B^{s+1}_{q,1}}\\
&\quad + \|\rho_0+\theta_2\|_{B^{s+1}_{q,1}}\|\BI-\BA_{\bu_2}\|_{B^{s+1}_{q,1}}
\|\nabla(\bu_1-\bu_2)\|_{B^{s+1}_{q,1}}\}\\
&\quad\leq C(\|\pd_t(\theta_1-\theta_2)\|_{L_1((0, T), B^{s+1}_{q,1})}\|\bu_1\|_{L_1((0, T), B^{s+1}_{q,1})}
\|\nabla\bu_1\|_{B^{s+1}_{q,1}}\\
&\quad + (\|\rho_0\|_{B^{s+1}_{q,1}}+\|\pd_t\theta_2\|_{L_1((0, T), B^{s+1}_{q,1})})
\|\bu_1-\bu_2\|_{L_1((0, T), B^{s+2}_{q,1}}
\|\nabla \bu_1\|_{B^{s+1}_{q,1}}
\\
&\quad +(\|\rho_0\|_{B^{s+1}_{q,1}}+\|\pd_t\theta_2\|_{L_1((0, T), B^{s+1}_{q,1})})
\|\nabla\bu_2\|_{L_1((0, T), B^{s+1}_{q,1})}\|\nabla(\bu_1-\bu_2)\|_{B^{s+1}_{q,1}}).
\end{align*}
Using the conditions: $E_T(\theta_i, \bu_i) \leq \omega$ ($i=1,2$), we have
$$\|F(\theta_1+\rho_0, \bu_1) - F(\theta_2+\rho_0, \bu_2)\|_{L_1((0, T), B^{s+1}_{q,1})}
\leq C(\omega+\omega^2)E_T(\theta_1-\theta_2, \bu_1-\bu_2),
$$
where $C$ depends on $\|\eta_0\|_{B^{s+1}_{q,1}}$. In fact, we estimate 
$$\|\rho_0\|_{B^{s+1}_{q,1}}+\|\pd_t\theta_2\|_{L_1((0, T), B^{s+1}_{q,1})}
\leq \|\rho_0-\eta_0\|_{B^{s+1}_{q,1}} + \|\eta_0\|_{B^{s+1}_{q,1}} + 
\|\pd_t\theta_2\|_{L_1((0, T), B^{s+1}_{q,1})}
\leq 2\omega + \|\eta_0\|_{B^{s+1}_{q,1}}.
$$
Summing up, we have obtained
\begin{equation}\label{diff:4}
\|\BF\|_{L_1((0, T), B^{s+1}_{q,1})} 
\leq C(\omega+ \omega^2)E_T(\theta_1-\theta_2, \bu_1-\bu_2)
\end{equation}
for some constant $C$ depending on $\|\eta_0\|_{B^{s+1}_{q,1}}$. \par 

Now, we treat $\BG$.  First, we estimate $\tilde\bG(\theta_1, \bu_1) - \tilde\bG(\theta_2, \bu_2)$.
Write
\begin{align*}
&\tilde\bG(\theta_1, \bu_1) - \tilde\bG(\theta_2, \bu_2) 
 = (\eta_0-\rho_0)\pd_t(\bu_1-\bu_2) \\
&+ \nabla (P(\rho_0+\theta_1) - P(\rho_0) - P'(\eta_0)\theta_1
-(P(\rho_0+\theta_2) - P(\rho_0) - P'(\eta_0)\theta_2))\\
& = (\eta_0-\rho_0)\pd_t(\bu_1-\bu_2) +
\nabla\{ \int^1_0P''(\eta_0 + \ell(\rho_0 - \eta_0))\,\d\ell
(\rho_0-\eta_0)(\theta_1-\theta_2) \}\\
&+\nabla\{\int^1_0(1-\ell)(P''(\rho_0+\ell\theta_1) - P''(\rho_0+\ell\theta_2))\,\d\ell \theta_1^2
+ \int^1_0(1-\ell)(P''(\rho_0+\ell\theta_2)\,\d\ell (\theta_1^2-\theta_2^2) \}.
\end{align*}
Writing $\eta_0 + \ell(\rho_0-\eta_0) = \eta_0 + (1-\ell)(\eta_0-\eta_0)
+ \ell(\rho_0-\eta_0)$, using \eqref{domain:4.1} and \eqref{assump:0}, we may assume that 
$$\rho_1/2-\rho_* < \tilde\eta_0 + \ell(\rho_0-\eta_0)
< 2\rho_2-\rho_*$$
for any $\ell \in (0, 1)$, and so we write
\begin{align*}
\int^1_0P''(\eta_0 + \ell(\rho_0 - \eta_0))\,\d\ell
= \int^1_0\int^1_0 P'''(\rho_* + m(\tilde\eta_0 + 
 \ell(\rho_0 - \eta_0))(\tilde\eta_0 + 
 \ell(\rho_0 - \eta_0))\,\d\ell\d m + P''(\rho_*).
\end{align*}
Thus,  by Lemmas \ref{lem:APH} and \ref{lem:Hasp} and \eqref{assump:1}, we have
\begin{align*}
&\|\nabla( \int^1_0P''(\eta_0 + \ell(\rho_0 - \eta_0))\,\d\ell
(\rho_0-\eta_0)(\theta_1-\theta_2))\|_{B^s_{q,1}}
\leq C(\rho_*, \|\tilde\eta_0\|_{B^{s+1}_{q,1}})\omega\|\theta_1-\theta_2\|_{B^{s+1}_{q,1}}.
\end{align*}
Write
$$\rho_0 + \ell\theta_2 + m(\rho_0 + \ell\theta_1-(\rho_0+\ell\theta_2))
= \eta_0 + (\rho_0-\eta_0) + \ell\theta_2 + m\ell(\theta_1-\theta_2).
$$
Since 
\begin{align*}
\|(\rho_0-\eta_0) + \ell\theta_2 + m\ell(\theta_1-\theta_2)\|_{L_\infty}
&\leq C(\|\rho_0-\eta_0\|_{B^{s+1}_{q,1}}
+\ell(1-m)\|\theta_2\|_{B^{s+1}_{q,1}} + \ell m\|\theta_1\|_{B^{s+1}_{q,1}})\\
&\leq C(\omega+\sum_{i=1,2}\|\pd_t\theta_i\|_{L_1((0, T), B^{s+1}_{q,1})})
\leq C\omega,
\end{align*}
we may assume that 
$$\rho_1/2-\rho_* < \tilde\eta_0 + (\rho_0-\eta_0)+\ell\theta_2
+ m\ell(\theta_1-\theta_2) < 2\rho_2-\rho_*,$$
and so, we write
\begin{align*}
&\int^1_0(1-\ell)(P''(\rho_0+\ell\theta_1)-P''(\rho_0 + \ell\theta_2))\,\d\ell\theta_1^2\\
&= \int^1_0\int^1_0(1-\ell)P'''(\rho_0+\ell\theta_2+m\ell(\theta_1-\theta_2))
(\theta_1-\theta_2)\,\d\ell\,\d m\, \theta_1^2\\
& = \int^1_0\int^1_0\int^1_0(1-\ell)P''''(\rho_*+n(\tilde\eta_0+(\rho_0-\eta_0)+\ell\theta_2
+m\ell(\theta_1-\theta_2))) \\
&\qquad \times(\tilde\eta_0+(\rho_0-\eta_0)+\ell\theta_2+m\ell(\theta_1-\theta_2))
\,\d\ell\d m\d n\,(\theta_1-\theta_2)\theta_1^2
+ \frac12P'''(\rho_*)(\theta_1-\theta_2)\theta_1^2.
\end{align*}
By Lemmas \ref{lem:APH} and \ref{lem:Hasp}, and \eqref{theta:1}, 
we have
\begin{align*}
&\|\nabla(\int^1_0(1-\ell)(P''(\rho_0+\ell\theta_1)-P''(\rho_0 + \ell\theta_2))\,\d\ell\theta_1^2)
\|_{B^s_{q,1}}\\
&\quad \leq C(\rho_*, \|\tilde\eta_0\|_{B^{s+1}_{q,1}})\|\pd_t\theta_1\|_{L_1((0, T), B^{s+1}_{q,1}}^2
\|\theta_1-\theta_2\|_{B^{s+1}_{q,1}}.
\end{align*}
Concerning the last term, we write $\rho_0 + \ell\theta_2 = \eta_0 + (\rho_0-\eta_0) + \ell\theta_2$.  Since
$$\|\rho_0-\eta_0 + \ell\theta_2\|_{L_\infty} \leq C(\|\rho_0-\eta_0\|_{B^{s+1}_{q,1}}+
\|\theta_2\|_{B^{s+1}_{q,1}})
\leq C(\|\rho_0-\eta_0\|_{B^{s+1}_{q,1}} + \|\pd_t\theta_2\|_{L_1((0, T), B^{s+1}_{q,1})})
\leq C\omega,$$
choosing $\omega>0$ small enough, we may assume that 
$$\rho_1/2-\rho_* < \tilde\eta_0 + (\rho_0-\eta_0) + \ell\theta_2 < 2\rho_2-\rho_*$$
for any $\ell \in (0, 1)$.  Thus, writing
\begin{align*}
&\int^1_0(1-\ell)P''(\rho_0 + \ell\theta_2)\,\d\ell\,(\theta_1^2-\theta_2^2)\\
&= \Bigl\{\frac12P''(\rho_*)
+ \int^1_0\int^1_0(1-\ell)P'''(\rho_*+m(\tilde\eta_0 + \rho_0-\eta_0+\ell\theta_2))
(\tilde\eta_0 + \rho_0-\eta_0+\ell\theta_2)\,\d\ell\d m\}\\
&\qquad\times(\theta_1-\theta_2)(\theta_1+\theta_2),
\end{align*}
By Lemmas \ref{lem:APH} and \ref{lem:Hasp}, and \eqref{theta:1}, we have
\begin{align*}
&\|\nabla(\int^1_0(1-\ell)P''(\rho_0 + \ell\theta_2)\,\d\ell\,(\theta_1^2-\theta_2^2))
\|_{B^s_{q,1}} \\
&\quad \leq C(\rho_*, \|\tilde\eta_0\|_{B^{s+1}_{q,1}})
(\|\tilde\eta_0\|_{B^{s+1}_{q,1}}+\|\rho_0-\eta_0\|_{B^{s+1}_{q,1}}
+\|\pd_t\theta_2\|_{L_1((0, T),B^{s+1}_{q,1})})\\
&\qquad\times (\|\pd_t\theta_1\|_{L_1((0, T),B^{s+1}_{q,1})}
+ \|\pd_t\theta_2\|_{L_1((0, T),B^{s+1}_{q,1})})\|\theta_1-\theta_2\|_{B^{s+1}_{q,1}}.
\end{align*}
Summing up, we have obtained 
\begin{align*}
&\|\tilde\bG(\theta_1, \bu_1) - \tilde\bG(\theta_2, \bu_2) \|_{L_1((0, T), B^s_{q,1}} 
\leq C\|\eta_0-\rho_0\|_{B^{N/q}_{q,1}}\|\pd_t(\bu_1-\bu_2)\|_{L_1((0, T), B^s_{q,1}} \\
&\quad  + C(\rho_*, \|\eta_0\|_{B^{s+1}_{q,1}})(\omega 
+ \|\pd_t\theta_1\|_{L_1((0, T), B^{s+1}_{q,1})}^2
+ \sum_{i=1}^2\|\pd_t\theta_i\|_{L_1((0, T), B^{s+1}_{q,1})})\|\theta_1-\theta_2\|_{L_1((0, T), B^{s+1}_{q,1})}.
\end{align*}
Since $E_T(\theta_i, \bu_i) \leq \omega$, using \eqref{semi:1}, we have
$$\|\tilde\bG(\theta_1, \bu_1) - \tilde\bG(\theta_2, \bu_2) \|_{B^s_{q,1}} 
\leq C(\omega+\omega^2)E_T(\theta_1-\theta_2, \bu_1-\bu_2).
$$
Finally, we estimate $\bG(\rho_0+\theta_1, \bu_1)-\bG(\rho_0+\theta_2, \bu_2)$. 
We write
\begin{align*}
&\bG(\rho_0+\theta_1, \bu_1)-\bG(\rho_0+\theta_2, \bu_2)\\
&= ((\BA_{\bu_2}^\top)^{-1}-(\BA_{\bu_1}^\top)^{-1})(\rho_0+\theta_1)\pd_t\bu_1
+ (\BI-(\BA_{\bu_2}^\top)^{-1})(\theta_1-\theta_2)\pd_t\bu_1 \\
&+ (\BI-(\BA_{\bu_2}^\top)^{-1})(\rho_0+\theta_2) \pd_t(\bu_1-\bu_2)\\
& + \alpha ((\BA_{\bu_1}^\top)^{-1}-(\BA_{\bu_2}^\top)^{-1})\dv(\BA_{\bu_1}\BA_{\bu_1}^\top:\nabla\bu_1)
+ \alpha((\BA_{\bu_2}^\top)^{-1}-\BI)\dv((\BA_{\bu_1}\BA_{\bu_1}^\top-\BA_{\bu_2}\BA_{\bu_2}^\top):\nabla\bu_1)\\
&+  \alpha((\BA_{\bu_2}^\top)^{-1}-\BI)\dv(\BA_{\bu_2}\BA_{\bu_2}^\top:\nabla(\bu_1-\bu_2))
+ \alpha\dv((\BA_{\bu_1}-\BA_{\bu_2})(\BA_{\bu_1}^\top-\BI):\nabla\bu_1) \\
& + \alpha\dv( \BA_{\bu_2}(\BA_{\bu_1}^\top-\BA_{\bu_2}^\top):\nabla\bu_1) 
+ \alpha\dv (\BA_{\bu_2}(\BA_{\bu_2}^\top-\BI):\nabla(\bu_1-\bu_2) )\\
&+ \beta\nabla((\BA_{\bu_1}^\top - \BA_{\bu_2}^\top):\nabla\bu_1)
+ \beta\nabla((\BA_{\bu_2}^\top-\BI):\nabla(\bu_1-\bu_2)).
\end{align*}
Employing the similar argument to the proof of \eqref{diff:4}, we have
\begin{align*}
\|(\BA_{\bu_1}^\top)^{-1} -(\BA_{\bu_2}^\top)^{-1}\|_{B^{N/q}_{q,1}} 
&\leq C\|\nabla(\bu_1-\bu_2)\|_{L_1((0, T), B^{N/q}_{q,1})}, \quad
\|(\BI-(\BA_{\bu_i})^\top\|_{B^{N/q}_{q,1}} \leq C\|\nabla \bu_i\|_{L_1((0, T), B^{N/q}_{q,1})}, \\
\|\BA_{\bu_1}^\top -\BA_{\bu_2}^\top\|_{B^{N/q}_{q,1}} 
&\leq C\|\nabla(\bu_1-\bu_2)\|_{L_1((0, T), B^{N/q}_{q,1})}, \quad 
\|\BA_{\bu_i}\BA_{\bu_i}^\top-\BI\|_{B^{s+1}_{q,1}}\leq C\|\nabla\bu_i\|_{L_1((0, T), B^{s+1}_{q,1})}, \\
	\|\BA_{\bu_1}\BA_{\bu_1}^\top - \BA_{\bu_2}\BA_{\bu_2}^\top\|_{B^{N/q}_{q,1}}
&\leq  C\|\nabla(\bu_1-\bu_2)\|_{L_1((0, T), B^{N/q}_{q,1})}.
\end{align*}
Therefore, by Lemmas \ref{lem:APH} and \ref{lem:Hasp}, we have
\begin{align*}
&\|\bG(\rho_0+\theta_1, \bu_1)-\bG(\rho_0+\theta_2, \bu_2)\|_{B^s_{q,1}} \\
& \leq C\{\|(\BA_{\bu_2}^\top)^{-1}-(\BA_{\bu_1}^\top)^{-1}\|_{B^{N/q}_{q,1}}
(\|\rho_0\|_{B^{N/q}_{q,1}} + \|\pd_t\theta\|_{L_1((0, T), B^{N/q}_{q,1})})
\|\pd_t\bu_1\|_{B^s_{q,1}} \\
&+ \|\nabla\bu_2\|_{L_1((0, T), B^{N/q}_{q,1})}\|\pd_t(\theta_1-\theta_2)\|_{L_1((0, T), 
B^{N/q}_{q,1})}\|\pd_t\bu_1\|_{B^s_{q,1}} \\
&  + \|\nabla\bu_2\|_{L_1((0, T), B^{N/q}_{q,1})}(\|\rho_0\|_{B^{N/q}_{q,1}}+
\|\pd_t\theta_2\|_{L_1((0, T), B^{N/q}_{q,1})})\|\pd_t(\bu_1-\bu_2)\|_{B^s_{q,1}} \\
&+ \|(\BA_{\bu_2}^\top)^{-1}-(\BA_{\bu_1}^\top)^{-1}\|_{B^{N/q}_{q,1}}
(1+\|\nabla\bu_1\|_{B^{s+1}_{q,1}})\|\bu_1\|_{B^{s+1}_{q,1}} \\
&+ \|\nabla\bu_2\|_{L_1((0, T), B^{N/q}_{q,1})}\|\BA_{\bu_1}\BA_{\bu_1}^\top
-\BA_{\bu_2}\BA_{\bu_2}^\top\|_{B^{s+1}_{q,1}}\|\nabla\bu_1\|_{B^{s+1}_{q,1}}
\\
&+  \|\nabla\bu_2\|_{L_1((0, T), B^{N/q}_{q,1})}
(1+ \|\nabla\bu_2\|_{L_1((0, T), B^{N/q}_{q,1})})\|\nabla(\bu_1-\bu_2) \|_{B^{s+1}_{q,1}} \\
&+ \|\nabla(\bu_1-\bu_2)\|_{L_1((0, T), B^{s+1}_{q,1})}\|\nabla\bu_1\|_{L_1((0, T), B^{s+1}_{q,1})}
\|\nabla\bu_1\|_{B^{s+1}_{q,1}} \\
& + (1+\|\nabla\bu_2\|_{L_1((0, T), B^{s+1}_{q,1})})\|\nabla(\bu_1-\bu_2)\|_{L_1((0, T), B^{s+1}_{q,1})}
\|\nabla\bu_1\|_{B^{s+1}_{q,1}} \\
&+ (1+\|\nabla\bu_2\|_{L_1((0, T), B^{s+1}_{q,1})})\|\nabla\bu_2\|_{L_1((0, T), B^{s+1}_{q,1})}
\|\nabla(\bu_1-\bu_2)\|_{B^{s+1}_{q,1}} \\
&+ \|\nabla(\bu_1 -\bu_2)\|_{L_1((0, T), B^{s+1}_{q,1})}\|\nabla\bu_1\|_{B^{s+1}_{q,1}} 
+ \|\nabla\bu_2\|_{L_1((0, T), B^{s+1}_{q,1})}\|\nabla(\bu_1-\bu_2)\|_{B^{s+1}_{q,1}}\}.
\end{align*}
We have $\|\rho_0\|_{B^{s+1}_{q,1}} \leq \|\rho_0-\eta_0\|_{B^{s+1}_{q,1}} + \|\tilde\eta_0\|_{B^{s+1}_{q,1}}
\leq C\omega + \|\tilde\eta_0\|_{B^{s+1}_{q,1}}$. 
Thus, we have
$$\|\bG(\rho_0+\theta_1, \bu_1)-\bG(\rho_0+\theta_2, \bu_2)\|_{L_1(0, T), B^s_{q,1})} 
\leq C(\rho_*, \|\tilde\eta_0\|_{B^{s+1}_{q,1}})(\omega+\omega^2)E_T(\theta_1-\theta_2,
\bu_1-\bu_2).$$
Summing up, we have obtained \eqref{diff:3}. \par
Combining \eqref{diff:2} and \eqref{diff:3} yields 
$$E_T(\eta_1-\eta_2, \bw_1-\bw_2) \leq Ce^{\gamma T}(\omega+\omega^2)
E_T(\theta_1-\theta_2, \bu_1,\bu_2).$$
Thus, choosing  $\omega>0$ and $T>0$ so small that $Ce(\omega+\omega^2) \leq 1/2$
and $\gamma T \leq 1$, we have 
$$E_T(\eta_1-\eta_2, \bw_1-\bw_2) \leq (1/2)
E_T(\theta_1-\theta_2, \bu_1,\bu_2),
$$
which shows that 
$\Phi$ is a contraction map from $S_{T, \omega}$ into itself. Therefore, by the Banach fixed
point theorem, $\Phi$ has a unique fixed point $(\eta, \bw) \in S_{T, \omega}$.
In \eqref{st:2}, setting $(\eta, \bw) = (\theta, \bu)$ and recalling 
$\rho = \rho_0+\theta$ and $\tilde\bG(\theta, \bu) = (\eta_0-\rho_0-\theta)\pd_t\bu
- \nabla(P(\rho_0+\theta)-P(\rho_0)-P'(\eta_0)\theta)$, we see that $\theta$ and $\bu$ satisfy
equations:
\begin{equation}\label{st:4}\left\{\begin{aligned}
\pd_t\theta+\eta_0\dv\bu = (\eta_0-\rho_0-\theta)\dv\bu +  F(\rho_0+\theta, \bu)& 
&\quad&\text{in $\HS\times(0, T)$}, \\
\eta_0\pd_t\bu - \alpha\Delta \bu  -\beta\nabla\dv\bu
+  \nabla(P'(\eta_0) \theta) = -\nabla P(\rho_0) + \bG(\rho_0+\theta, \bu)
- \tilde\bG(\theta, \bu)
& &\quad&\text{in $\HS\times(0, T)$}, \\
\bu|_{\pd\HS} =0, \quad (\eta, \bu)|_{t=0} = (0, \bu_0)
& &\quad&\text{in $\HS$}.
\end{aligned}\right.\end{equation}
Thus, setting $\rho = \rho_0 + \theta$, from \eqref{st:4} it follows that 
$\rho$ and $\bu$ satisfy equations \eqref{ns:2}.  Moreover, 
$(\rho, \bu)$ belongs to $S_{T,\omega}$, which completes the proof of 
Theorem \ref{thm:2}.

{\bf A proof of Theorem \ref{thm:1}.}~ As was mentioned at the beginning of 
Subsec. \ref{sec.1.1}, $y= X_\bu(x, t)$ is a $C^1$ diffeomorphism from $\Omega$ onto itself
for any $t \in (0, T)$, because $\bu \in L_1((0, T), B^{s+2}_{q,1}(\Omega)^N)$.
Let $x = X_\bu^{-1}(y, t)$ be the inverse of $X_\bu$. For any function 
$F \in B^s_{q,1}(\HS)$, $1 < q < \infty$, $s \in \BR$, it follow from the chain rule that 
$$\|F\circ X_\bu^{-1}\|_{B^s_{q,1}(\HS)} \leq C\|F\|_{B^s_{q,1}(\HS)}$$
with some constant $C>0$ (cf. Amann \cite[Theorem 2.1]{Amann00}). 
Let $(\rho, \bv) = (\theta, \bu)\circ X_\bu^{-1}$ and $\BA_\bu = (\nabla_yX_\bu)^{-1}$.
Let $\BA_\bu^\top = (A_{jk})$.  There holds
\begin{gather*}
\nabla_y(\rho, \bv) = (\BA_\bu^\top \nabla_x(\theta, \bu))\circ X_\bu^{-1}, \\
\pd_{y_j}\pd_{y_k}\bv = \sum_{\ell, \ell'}
A_{j\ell}\pd_{y_\ell}(A_{k\ell'}\pd_{y_{\ell'}}\bu))\circ X_\bu^{-1}
\quad(j, k =1, \ldots, N).
\end{gather*}
Hence,   
we rely on the relation:
$$\pd_t(\rho, \bv) = \pd_t(\theta, \bu)\circ X_\bu^{-1}
-((\bu\circ X_\bu^{-1})\cdot\nabla_y)(\rho, \bv),
$$
concerning the time derivative of $\rho$ and $\bv$. 
Therefore, by Theorem \ref{thm:2} and Lemma \ref{lem:APH}, 
we arrive at \eqref{main.reg}.  This completes the proof of Theorem \ref{thm:1}. 


\end{document}